\documentclass{article}
\usepackage{amsmath,amssymb,amsfonts,amsthm,bm}

\usepackage{graphicx,xcolor}\graphicspath{{figures/}}
\usepackage{subfigure}
\usepackage{multirow}
\usepackage{longtable}
\usepackage{chapterbib}
\usepackage{hyperref}
\usepackage{algorithm}
\usepackage[framemethod=tikz]{mdframed}  
\usepackage{algorithmicx}
\usepackage{array}
\usepackage[ntheorem]{empheq}
\usepackage{algpseudocode}

\usepackage{amsmath,latexsym,amssymb,amsfonts,amsbsy}
\usepackage{graphicx}
\usepackage{cite}
\usepackage{color}

\newtheorem{theorem}{Theorem}[section]

\newtheorem{lemma}{Lemma}[section]

\newtheorem{corollary}{Corollary}[section]

\begin{document}
\title{The sparse Kaczmarz method with surrogate hyperplane for the regularized basis pursuit problem}
\author{
Ze Wang\\
School of Mathematical Sciences, Tongji University, \\
N.O. 1239, Siping Road, Shanghai, 200092 \\
Email: math\_wangze@tongji.edu.cn\\
and\\
Jun-Feng Yin\\
School of Mathematical Sciences, \\
Key Laboratory of Intelligent Computing and Applications (Ministry of Education),\\
Tongji university, N.O. 1239, Siping Road, Shanghai, 200092 \\
Email:yinjf@tongji.edu.cn\\
and\\
Ji-Chen Zhao\\
School of Mathematical Sciences, Tongji University, \\
N.O. 1239, Siping Road, Shanghai, 200092 \\
Email:jczhao@tongji.edu.cn\\}

\maketitle
\footnote{This work is supported by the National Natural Science Foundation of China (Grant No. 11971354).}

\begin{abstract}
{\large } The Sparse Kaczmarz method is a famous and widely used iterative method for solving the regularized basis pursuit problem. A general scheme of the surrogate hyperplane sparse Kaczmarz method is proposed. In particular, a class of residual-based surrogate hyperplane sparse Kaczmarz method is introduced and the implementations are well discussed. Their convergence theories are proved and the linear convergence rates are studied and compared in details. Numerical experiments verify the efficiency of the proposed methods.
\end{abstract}


\section{Introduction}
\quad Consider solving the following regularized basis pursuit problem: 
\begin{equation}
	\label{equ1}
	\min_{x\in \mathbb{R}^n} \lambda\|x\|_1+\frac{1}{2}\|x\|_2^2, \ \ \text{s.t.} \ \ Ax=b,
\end{equation}
where $A\in\mathbb{R}^{m\times n }$, $b\in\mathbb{R}^m$ and $\lambda>0$, which arises in various fields of scientific computing, including compressed sensing \cite{donoho2006compressed}, image processing \cite{candes2006robust} and machine learning \cite{byrd2012sample}. 

The Kaczmarz method is a classical and efficient projection iteration method for solving consistent linear equations and its randomized version has attracted much more attention \cite{strohmer2009randomized,zeng2023adaptive,han2024randomized}. Several representative randomized Kaczmarz methods and their extensions were summarized by Bai and Wu, and their asymptotic convergence theories were proved in \cite{bai2023randomized}. By theoretically analyzing and numerically experimenting several criteria adopted for selecting the working row, sharper upper bounds for the convergence rates of Kaczmarz-type methods was given in \cite{bai2023convergence}. Furthermore, it was verified that sampling without replacement and using quasirandom numbers are the fastest techniques for solving consistent systems \cite{ferreira2024survey}. Recently, the Kaczmarz-type methods have been generalized to solve the regularized basis pursuit problem \eqref{equ1} and its convergence theory was established in \cite{lorenz2014sparse, lorenz2014linearized}.

The randomized sparse Kaczmarz method was introduced by Stefania Petra and its sublinear convergence rate was analyzed in terms of the dual objective function by the relation with the coordinate gradient descent method  \cite{petra2015randomized}. However, the linear convergence rate was obtained by smoothing the primal objective function, resulting in a modification of the iteration process. Since the randomized sparse Kaczmarz method without smoothing converges linearly in expectation was proved in \cite{schopfer2019linear}, numerous variants of the randomized sparse Kaczmarz method were further proposed and studied, including online randomized sparse Kaczmarz method \cite{lei2018learning}, greedy randomized sparse Kaczmarz method \cite{wang2021}, randomized average sparse Kaczmarz method. Moreover, some acceleration techniques have also been applied to the sparse Kaczmarz methods, such as microkicking \cite{lunglmayr2017microkicking}, restarting\cite{tondji2023acceleration}, sketching \cite{yuan2022adaptively}.
The sparse Kaczmarz method was extended to solve linear inverse problems with independent noise exactly and systems of nonlinear equations, known as adaptive  Bregman-Kaczmarz method \cite{tondji2023adaptive} and nonlinear Bregman-Kaczmarz method \cite{gower2023nonlinear}, respectively. For more recent developments on the sparse Kaczmarz method, we refer to \cite{lorenz2023minimal,yun2023fast} and references therein.

In order to accelerate the convergence speed of the sparse Kaczmarz method, a surrogate hyperplane sparse Kaczmarz method is proposed. In particular, two implementations of the residual-based surrogate hyperplane sparse Kaczmarz method are introduced by making use of the residual in each iteration. The convergence theory of the proposed methods is proved and their linear convergence rates are studied in details. Numerical experiments further verified the efficiency of the new methods for the application to random matrices, SuiteSparse matrices and image restoration.

The organization of the paper is as follows. In section \ref{sec:2}, the surrogate hyperplane sparse Kaczmarz method is constructed and well studied. In section \ref{sec:3}, two implementations of residual-based surrogate hyperplane sparse Kaczmarz method are introduced and the convergence theories are established. In section \ref{sec:4}, numerical experiments demonstrate the effectiveness of the new methods. Finally, some  conclusions are drawn in section \ref{sec:5}.

\section{The surrogate hyperplane sparse Kaczmarz method}
\label{sec:2}
\quad In this section, after giving concepts and properties of convex functions, a surrogate hyperplane sparse Kaczmarz method is constructed. It is proved that the Bregman distance of the iteration sequences is momotonically non-increasing.

Let $\operatorname{supp}(\hat{x})=\left\{j \in\{1, \ldots, n\} \mid \hat{x}_{j} \neq 0\right\}$, where $\hat{x}$ is the unique solution of the regularized basis pursuit problem (\ref{equ1}). Denote $A_J$ as the matrix that is formed by the columns of $A$ indexed by $J$, define 
\begin{equation}
	\label{equ8}
	\tilde{\sigma}_{\min }(A)=\min \left\{\sigma_{\min }\left(A_{J}\right) \mid J \subset\{1, \ldots, n\}, A_{J} \neq 0\right\}.
\end{equation}
Assuming that $b\neq 0$, then $\hat{x}\neq 0$ and hence
\begin{equation}
	\label{equ9}
	|\hat{x}|_{\min }=\min \left\{\left|\hat{x}_{j}\right| \mid j \in \operatorname{supp}(\hat{x})\right\}>0.
\end{equation}

The conjugate function of $f: \mathbb{R}^{n} \rightarrow \mathbb{R}$ at $y\in\mathbb{R}^n$ is defined as
$$
f^{*}(y):=\sup _{x \in \mathbb{R}^{n}}\{\langle y, x\rangle-f(x)\}.
$$
Since $f$ is assumed to be finite everywhere, it is also continuous. Denote $\partial f(x)$ as the subdifferential of $f$ at $x$,
$$
\partial f(x)=\left\{x^{*} \in \mathbb{R}^{n} \mid f(y) \geq f(x)+\left\langle x^{*}, y-x\right\rangle \right\},
$$
which is nonempty, compact and convex. If $f$ is differentiable, denote the $\nabla f(x)$ is the gradient of $f(x)$, then $\partial f(x)= \{\nabla f(x)\}$.

The convex function $f$ is called $\alpha$-strongly convex, if there is some $\alpha>0$ such that for all $x,y$ and $x^*\in \partial f(x)$, it is satisfied that
$$
f(y) \geq f(x)+\left\langle x^{*}, y-x\right\rangle+\frac{\alpha}{2} \cdot\|y-x\|_{2}^{2}.
$$

If $f$ is $\alpha$-strongly convex, then the conjugate function $f^*$ is differentiable with $1/\alpha$-Lipschitz-continuous gradient, i.e.
$$
\left\|\nabla f^{*}\left(x^{*}\right)-\nabla f^{*}\left(y^{*}\right)\right\|_{2} \leq \frac{1}{\alpha} \cdot\left\|x^{*}-y^{*}\right\|_{2},
$$
which implies that 
$$
f^{*}\left(y^{*}\right) \leq f^{*}\left(x^{*}\right)+\left\langle\nabla f^{*}\left(x^{*}\right), y^{*}-x^{*}\right\rangle+\frac{1}{2 \alpha}\left\|x^{*}-y^{*}\right\|_{2}^{2}.
$$

The Bregman distance $D_f^{x^*}(x,y)$ between $x,y$ with respect to $f$ and a subgradient $x^*\in \partial f(x)$ is defined as 
	$$
	D_{f}^{x^{*}}(x, y):=f(y)-f(x)-\left\langle x^{*}, y-x\right\rangle.
	$$
   From the Proposition 11.3 in \cite{Rockafellar1998VariationalA}, when $x^*\in \partial f(x)$, it is known that $f^*(x^*)=\langle x^*,x\rangle-f(x)$, therefore
	$$
	D_{f}^{x^{*}}(x, y)=f^{*}\left(x^{*}\right)-\left\langle x^{*}, y\right\rangle+f(y).
	$$
According to the concepts and properties of $\alpha$-strongly convex function, the following two lemmas provide the relationship among Bregman distance, Euclidean distance and residual.

\begin{lemma}\cite{lorenz2014linearized}
	\label{lem1}
	Let $f$ be continuous and strongly convex with constant $\alpha>0$. For all $x,y\in \mathbb{R}^{n}$ and $x^*\in \partial f(x)$, it is satisfied that 
	$$
	D_f^{x^{*}}(x, y) \geq \frac{\alpha}{2}\|x-y\|_{2}^{2} \geq 0,
	$$
	and
	$
	D_f^{x^{*}}(x, y)=0$ if and only if  $ x=y$.
\end{lemma}

Since $f(x)=\lambda\|x\|_1+\frac{1}{2}\|x\|_2^2$ is 1-strongly convex, it follows that
\begin{equation}
	\label{equ7}
	\frac{1}{2}\|x-y\|_{2}^{2}\leq D_f^{x^{*}}(x, y) .
\end{equation}

\begin{lemma}\cite{schopfer2019linear}
	\label{lem2}
	Let $\tilde{\sigma}_{\min }(A) \text { and }|\hat{x}|_{\min }$ be defined in $\eqref{equ8}$ and $\eqref{equ9}$, respectively. Then, for all $x\in\mathbb{R}^{n}$ with $\partial f(x) \cap \mathcal{R}\left(A^{T}\right) \neq \emptyset \text { and all } x^{*}=A^{T} y \in \partial f(x) \cap \mathcal{R}\left(A^{T}\right)$, it holds that
	\begin{equation}
		\label{equ10}
		D_{f}^{x^{*}}(x, \hat{x}) \leq \nu \|A x-b\|_{2}^{2},
	\end{equation}
	where $\nu=\frac{1}{\tilde{\sigma}_{\min }^{2}(A)} \cdot \frac{|\hat{x}|_{\min }+2 \lambda}{|\hat{x}|_{\min }} $, $\mathcal{R}\left(A^{T}\right)$ is the range of $A^T$ and $\tilde{\sigma}_{\min}(A)$ denotes the smallest nonzero singular value.
\end{lemma}

Denote $a_i^T$ is the $i$th row of $A$ and $b_i$ is the $i$th element of $b$. The Kaczmarz method project the current iteration onto the hyperplane 
	$$a_i^Tx=b_i,$$
where the row index $i$ is chosen cyclically or randomly. The surrogate hyperplane Kaczmarz method calculate the weight vector $\eta_k\in \mathbb{R}^m$ to linearly combine multiple hyperplanes of linear equations for projection \cite{wang2023surrogate}
$$
\eta_k^TAx=\eta_k^Tb.
$$

The sparse Kaczmarz method for solving the regularized basis pursuit problem \eqref{equ1} performs the update in each iteration by selecting a row index $i$,
$$
\begin{array}{l}x_{k+1}^{*}=x_{k}^{*}+ \frac{\left(b_i-a_{i}^Tx_k\right)}{\left\|a_{i}\right\|_{2}^{2}} a_{i}, \\ x_{k+1}=S_{\lambda}\left(x_{k+1}^{*}\right).\end{array}
$$
where the soft shrinkage function $S_{\lambda}(x)=\text{sign}(x)\max(|x|-\lambda,0)$, $\text{sign}(x)$ and $|x|$ represent the sign function and the absolute value of vector $x$ respectively.

Inspired by the sparse Kaczmarz mehtod, the surrogate hyperplane sparse Kaczmarz method for solving the regularized basis pursuit problem \eqref{equ1} is constructed in Algorithm \ref{agm1}.
\begin{algorithm}[!htbp]
	\begin{algorithmic}[1]
		\caption{The surrogate hyperplane sparse Kaczmarz method}
		\label{agm1}
		\State \textbf{Input:} $x_{0}=x_{0}^*=0$, $A$, $b$ and $\lambda>0$.
		\For{$k=0,1,2,\ldots$}
		\State Determine the weight vector $\eta_k$ to generate the surrogate hyperplane.
		\State Compute
		$$
		x_{k+1}^*=x_{k}^*+\frac{\eta_k^T(b-Ax_k)}{\left\|A^T\eta_k\right\|_{2}^{2}}A^T\eta_k.
		$$
		\State Compute
		$$
		x_{k+1}=S_{\lambda}(x^*_{k+1}).
		$$
		\EndFor
	\end{algorithmic}
\end{algorithm}

In particular, if $\eta_{k}$ is a Gaussian vector that satisfis the normal distribution, the surrogate hyperpalne sparse Kaczmarz method is an improvement of the Gaussian Kaczmarz method \cite{gower2015randomized} for solving the regularized basis pursuit problem (\ref{equ1}). 

If $\eta_k=e_i$, where $e_i$ is the $i$th column of the identity matrix, the surrogate hyperplane sparse Kaczmarz method is the original sparse Kaczmarz method \cite{lorenz2014sparse}. In \cite{schopfer2019linear}, it was proved that the randomized sparse Kaczmarz method converge in expection to the unique solution $\hat{x}$ of the regularized Basis Pursuit problem (\ref{equ1}) with a linear rate, it holds that 
$$
\mathbb{E}\left[D_{f}^{x_{k+1}^{*}}\left(x_{k+1}, \hat{x}\right)\right] \leq (1-\hat{q}) \cdot \mathbb{E}\left[D_{f}^{x_{k}^{*}}\left(x_{k}, \hat{x}\right)\right]
$$
where $\hat{q}=\frac{1}{2\nu \|A\|_F^2}$, $\nu=\frac{1}{\tilde{\sigma}_{\min }^{2}(A)} \cdot \frac{|\hat{x}|_{\min }+2 \lambda}{|\hat{x}|_{\min }}$, and the expection is taken with respect to the probability distribution $p_i=\left\|a_{i}\right\|_{2}^{2} /\|A\|_{F}^{2}$.

In the case where the weight vector $\eta_{k}$ is undetermined, the relationship between the Bregman distance of the current point to the true solution $\hat{x}$ and that of the next iteration is presented in the following lemma.

\begin{lemma}
	\label{thm0}
	Given an initial vector $x_0=x_0^*=0$, assuming that $b\neq 0$ is in the range $\mathcal{R}(A)$ of $A$. Then, the Bregman distance between the iteration sequence $\{x_k\}_{k=0}^{\infty}$ of the Algorithm $\ref{agm1}$ and the unique solution $\hat{x}$ of the regularized basis pursuit problem $\eqref{equ1}$ satisfies that 
	$$
		D_f^{x_{k+1}^*}(x_{k+1},\hat{x})\leq D_{f}^{x_k^*}(x_k,\hat{x})-\frac{1}{2}\cdot \frac{(\eta_{k}^T(b-Ax_k))^2}{\|A^T\eta_{k}\|_2^2}.
	$$
\end{lemma}

\begin{proof}
	From the definition for the Bregman distance and $\nabla f^*$ is 1-Lipschitz continuous, it follows that 
	$$
		\begin{aligned} D_f^{x_{k+1}^*}(x_{k+1},\hat{x})& =f^{*}(x_{k+1}^*)-\langle x_{k+1}^*,\hat{x}\rangle+f(\hat{x})\\ & \leq f^*(x_k^*)+\langle \nabla f^*(x_k^*),x_{k+1}^*-x_k^*\rangle +\frac{1}{2} \|x_{k+1}^*-x_{k}^*\|_2^2+f(\hat{x})-\langle x_{k+1}^*,\hat{x}\rangle \\
			& = D_{f}^{x_k^*}(x_k,\hat{x})+\langle x_{k+1}^*-x_k^*,x_k-\hat{x} \rangle +\frac{1}{2}\|x_{k+1}^*-x_k^*\|_2^2.\end{aligned}
	$$
	It is easily seen that $x^*_{k+1}-x^*_{k}=\frac{\eta_k^T(b-Ax_k)}{\left\|A^T\eta_k\right\|_{2}^{2}}A^T\eta_k$ in the Algorithm $\ref{agm1}$, so
	$$
	\begin{aligned}
		D_f^{x_{k+1}^*}(x_{k+1},\hat{x})&\leq D_{f}^{x_k^*}(x_k,\hat{x})+\left\langle \frac{\eta_k^T(b-Ax_k)}{\left\|A^T\eta_k\right\|_{2}^{2}}A^T\eta_k,x_k-\hat{x} \right\rangle +\frac{1}{2}\|\frac{\eta_k^T(b-Ax_k)}{\left\|A^T\eta_k\right\|_{2}^{2}}A^T\eta_k\|_2^2\\
		&=D_{f}^{x_k^*}(x_k,\hat{x})+\left\langle \frac{\eta_k^T(b-Ax_k)}{\left\|A^T\eta_k\right\|_{2}^{2}}\eta_k,Ax_k-b \right\rangle+\frac{1}{2}\|\frac{\eta_k^T(b-Ax_k)}{\left\|A^T\eta_k\right\|_{2}^{2}}A^T\eta_k\|_2^2\\
		&=D_{f}^{x_k^*}(x_k,\hat{x})-\frac{1}{2}\cdot \frac{(\eta_{k}^T(b-Ax_k))^2}{\|A^T\eta_{k}\|_2^2}.
	\end{aligned}
	$$
\end{proof}

It is observed from Lemma $\ref{thm0}$ that the last term $\frac{(\eta_{k}^T(b-Ax_k))^2}{\|A^T\eta_{k}\|_2^2}$ is nonnegative, so the Bregman distance between $x_k$ and $\hat{x}$ is monotonically non-increasing. In the following section, two specific implementations for computing $\eta_{k}$ are provided, and the convergence theorems for the corresponding methods are demonstrated by making use of Lemma $\ref{thm0}$.

\section{The Residual-based Surrogate hyperplane sparse Kaczmarz method }
\label{sec:3}
\quad In this section, by making use of the residuals to construct the weight vector $\eta_{k}$, the residual-based surrogate hyperplane Kaczmarz method and the partial residual-based surrogate hyperplane Kaczmarz method are proposed respectively. The convergence theories of these methods are well studied.

The residual-based surrogate hyperplane sparse Kaczmarz method is described as follows. 
\begin{algorithm}[H]
	\begin{algorithmic}[1]
		\caption{The residual-based surrogate hyperplane sparse Kaczmarz method}
		\label{agm2}
		\State \textbf{Input:} $x_{0}=x_{0}^*=0$, $A$, $b$ and $\lambda>0$.
        \For{$k=0,1,2,\ldots$}
		\State Compute the weight vector $\eta_k=b-Ax_k$.
		\State Compute
		$$
		x_{k+1}^*=x_{k}^*+\frac{\eta_k^T\eta_{k}}{\left\|A^T\eta_k\right\|_{2}^{2}}A^T\eta_k.
		$$
		\State Compute
		$$
		x_{k+1}=S_{\lambda}(x_{k+1}^*).
		$$
		\EndFor
	\end{algorithmic}
\end{algorithm}

 Due to the soft shrinkage function $S_{\lambda}(x)$, the residual-based surrogate hyperplane sparse Kaczmarz method is able to solve the sparse solution of the regularized basis pursuit problem \eqref{equ1}.

\begin{theorem}
	\label{thm1}
	 Given an initial vector $x_0=x_0^*=0$, assuming that $b\neq 0$ is in the range $\mathcal{R}(A)$ of $A$. The iteration sequence $\{x_k\}_{k=0}^{\infty}$ of the Algorithm $\ref{agm2}$ converges to the unique solution $\hat{x}$ of the regularized basis pursuit problem $\eqref{equ1}$, it holds that
	$$
	D_f^{x_{k+1}^*}(x_{k+1},\hat{x})\leq \left(1-q\right)D_{f}^{x_k^*}(x_k,\hat{x}),
	$$
	and
	$$
	\|x_k-\hat{x}\|_2\leq \left(1-q\right)^{\frac{k}{2}}\sqrt{\left(2\lambda \|\hat{x}\|_1+\frac{1}{2}\|\hat{x}\|_2^2\right)}.
	$$
	where $q =\frac{1}{2\nu \sigma^2_{\max}(A)}$, $\nu=\frac{1}{\tilde{\sigma}_{\min }^{2}(A)} \cdot \frac{|\hat{x}|_{\min }+2 \lambda}{|\hat{x}|_{\min }} $, and $\sigma_{\max}(A)$ denotes the maximum singular value of $A$.
\end{theorem}

\begin{proof}
	
	Since $\eta_k=b-Ax_k$ in the Algorithm \ref{agm2}, according to Lemma $\ref{lem2}$ and Lemma $\ref{thm0}$, it is obtained that
	$$
	\begin{aligned}
		D_f^{x_{k+1}^*}(x_{k+1},\hat{x})&\leq D_{f}^{x_k^*}(x_k,\hat{x})-\frac{1}{2}\cdot \frac{\|\eta_{k}\|^2_2\cdot \|\eta_{k}\|^2_2}{\|A^T\eta_{k}\|_2^2}\\
		& \leq D_{f}^{x_k^*}(x_k,\hat{x})-\frac{1}{2}\cdot \frac{\|b-Ax_k\|_2^2}{\sigma_{\max}^2(A)}\\
		& \leq \left(1-\frac{1}{2\nu \sigma_{\max}^2(A)}\right)D_{f}^{x_k^*}(x_k,\hat{x}).
	\end{aligned}
	$$
	Denote
	$$
	q=\frac{1}{2\nu \sigma_{\max}^2(A)},
	$$
	So
	$$
	D_f^{x_{k+1}^*}(x_{k+1},\hat{x})\leq \left(1-q\right)D_{f}^{x_k^*}(x_k,\hat{x}).
	$$
    which is concluded that by starting with the initial vector $x_0=x_0^*=0$,
	$$
	\|x_k-\hat{x}\|_2\leq \left(1-q\right)^{\frac{k}{2}}\sqrt{\left(2\lambda \|\hat{x}\|_1+\frac{1}{2}\|\hat{x}\|_2^2\right)}.
	$$
	 Therefore, if $b\neq 0$, then $\hat{x}\neq 0$, it is easy to know that 
	$$
	0<q=\frac{1}{2}\cdot\frac{\tilde{\sigma}_{\min}^2(A)}{\sigma_{\max}^2(A)} \cdot \frac{|\hat{x}|_{\min}}{|\hat{x}|_{\min}+2\lambda}\leq \frac{1}{2}<1
	$$
\end{proof}

It is easily found that the contraction factor $q$ of the residual-based surrogate hyperplane Kaczmarz method is a constant, so the iteration sequence $\{x_k\}_{k=0}^{\infty}$ converge to the unique solution $\hat{x}$ of the regularized basis pursuit problem ($\ref{equ1}$) linearly. 

Moreover, it is easily seen that the contraction factor of the residual-based surrogate hyperplane sparse Kaczmarz method is $q =\frac{1}{2\nu \sigma^2_{\max}(A)}$, while that of the randomized sparse Kaczmarz method is $\hat{q}=\frac{1}{2\nu \|A\|_F^2}$, it is obvious that $q\geq \hat{q}$, which shows that the covergence rate of the residual-based surrogate hyperplane sparse Kaczmarz method is faster than the randomized sparse Kaczmarz method theoretically.

In particular, it is not necessary to use entire residual information because certain residuals are closed to zero after several iterations. Therefore, considering to choose larger residuals based on an adaptive strategy, the partial residual-based surrogate hyperplane sparse Kaczmarz method is proposed, which is described in Algorithm $\ref{agm3}$.
\begin{algorithm}[!htbp]
	\begin{algorithmic}[1]
		\caption{The partial residual-based surrogate hyperplane sparse Kaczmarz method}
		\label{agm3}
		\State \textbf{Input:} $x_{0}=x_{0}^*=0$, $A$, $b$, $\theta\in[0,1]$and $\lambda>0$.
		\For{$k=0,1,2,\ldots$}
		\State Compute
		$$
					\epsilon_{k}=\frac{\theta}{\left\|b-A x_{k}\right\|_{2}^{2}} \max _{1 \leq i \leq m}\left\{\frac{\left|b_{i}-a_{i}^T x_{k}\right|^{2}}{\left\|a_{i}\right\|_{2}^{2}}\right\}+\frac{1-\theta}{\|A\|_{F}^{2}}.
		$$
		\State Determine the index set of positive integers
		$$
			\tau_{k}=\left\{i|| b_{i}-\left.a_{i} ^Tx_{k}\right|^{2} \geq \epsilon_{k}\left\|b-A x_{k}\right\|_{2}^{2}\left\|a_{i}\right\|_{2}^{2}\right\}.
		$$
		\State Compute the parameter
		$$
			\eta_{k}=\sum_{i \in \tau_{k}}\left(b_{i}-a_{i}^T x_{k}\right) e_{i}.
		$$
		\State Compute
		$$
			x^*_{k+1}=x^*_{k}+\frac{\eta_k^T(b-Ax_k)}{\left\|A^T\eta_k\right\|_{2}^{2}}A^T\eta_k.
		$$
		\State Compute
		$$
			x_{k+1}=S_{\lambda}(x_{k+1}^*).
		$$
		\EndFor
	\end{algorithmic}
\end{algorithm}
	
The Algorithm $\ref{agm3}$ covers the relaxed version of the fast deterministic block Kaczmarz method \cite{chen2022fast} without the last iteration, which is used for solving the consistent linear equations, but fails to solve the regularized basis pursuit problem \eqref{equ1}. By the straightforward computations, the iteration is rewritten as: 
$$
\begin{aligned} x^*_{k+1} & =x^*_{k}+\frac{\eta_{k}^{T}\left(b-A x_{k}\right)}{\left\|A^{T} \eta_{k}\right\|_{2}^{2}}\left(\sum_{i \in \tau_{k}}\left(b_{i}-a_{i}^T x_{k}\right)a_{i}\right) \\ & =x_{k}+\frac{\left\|\eta_{k}\right\|_{2}^{2}\left\|A_{\tau_{k}}\right\|_{F}^{2}}{\left\|A^{T} \eta_{k}\right\|_{2}^{2}}\left(\sum_{i \in \tau_{k}} \frac{\left\|a_{i}\right\|_{2}^{2}}{\left\|A_{\tau_{k}}\right\|_{F}^{2}} \frac{(b_{i}-a_{i}^T x_{k})}{\left\|a_{i}\right\|_{2}^{2}}a_{i}\right),\end{aligned}
$$
where $A_{\tau_{k}}$ stands for the row submatrix of $A$ indexed by $\tau_{k}$. By choosing $t_{k}=\frac{\left\|\eta_{k}\right\|_{2}^{2}\left\|A_{\tau_{k}}\right\|_{F}^{2}}{\left\|A^{T} \eta_{k}\right\|_{2}^{2}}$ and $\omega_{i}^{k}=\frac{\left\|a_{i}\right\|_{2}^{2}}{\left\|A_{\tau_{k}}\right\|_{F}^{2}}$, the partial residual-based surrogate hyperplane sparse Kaczmarz method iterates as
$$
x^*_{k+1}=x^*_{k}+t_{k}\left(\sum_{i \in \tau_{k}} \omega_{i}^{k} \frac{(b_{i}-a_{i}^T x_{k})}{\left\|a_{i}\right\|_{2}^{2}}a_{i}\right),
$$
which is an improved parallel version of the faster randomized block sparse Kaczmarz by averaging \cite{tondji2023faster}. 

When $\theta=1$ for the Algorithm $\ref{agm3}$, only one row index with the greatest residual is selected for iteration,  
	$$
	i=\max _{1 \leq i \leq m}\left\{\frac{\left|b_{i}-a_{i}^{T} x_{k}\right|^{2}}{\left\|a_{i}\right\|_{2}^{2}}\right\},
	$$
	then $\eta_k=(b_i-a_i^Tx_k)e_i$, the iteration of the partial residual-based surrogate hyperplane sparse Kaczamarz metheod is 
	$$
	x_{k+1}^*=x_k^*+\frac{(b-a_i^Tx_k)e_i^T(b-Ax_k)}{\|A^T(b_i-a_i^Tx_k)e_i\|_2^2}A^T(b_i-a_i^Tx_k)e_i=x_k^*+\frac{b_i-a_i^Tx_k}{\|a_i\|_2^2}a_i,
	$$
	which degenerates into the original sparse Kaczmarz method. 
	
	When $\theta=0$ for the Algorithm \ref{agm3}, the row index is computed by
		$$
		\frac{\left|b_{i}-a_{i}^{T} x_{k}\right|^{2} }{\|a_i\|_2^2}\geq \frac{\|b-Ax_k\|_2^2}{\|A\|_F^2},
		$$
		Since some residuals are closed to zero after several iterations, the number of row index in $\tau_k$ is almost equal to $m$, then
		$$
		b-Ax_k \approx \sum_{i \in \tau_{k}}\left(b_{i}-a_{i}^{T} x_{k}\right) e_{i}.
		$$
		which shows that the Algorithm \ref{agm2} and \ref{agm3} have similar convergence rate.
		
		In the following theorem, the linear convergence rate of the Algorithm \ref{agm3} is analyzed in details.

\begin{theorem}
	\label{thm2}
	 Given an initial vector $x_0=x_0^*=0$, assuming that $b\neq 0$ is in the range $\mathcal{R}(A)$ of $A$. Then, the iteration sequence $\{x_k\}_{k=0}^{\infty}$ of the Algorithm $\ref{agm3}$ converges to the unique solution $\hat{x}$ of the regularized basis pursuit problem $\eqref{equ1}$, it holds that
	$$
	D_f^{x_{k+1}^*}(x_{k+1},\hat{x})\leq \left(1-q_k\right)D_{f}^{x_k^*}(x_k,\hat{x}).
	$$
	and
    $$
	\|x_k-\hat{x}\|_2\leq \prod_{j=0}^{k}\left(1-q_j\right)^{\frac{k}{2}}\sqrt{\left(2\lambda \|\hat{x}\|_1+ \|\hat{x}\|_2^2\right)}.
	$$
	where $q_k=\frac{\epsilon_{k} \|A_{\tau_{k}}\|_F^2}{ 2\cdot \nu \sigma_{\max}^2(A_{\tau_{k}})}$, $\nu=\frac{1}{\tilde{\sigma}_{\min }^{2}(A)} \cdot \frac{|\hat{x}|_{\min }+2 \lambda}{|\hat{x}|_{\min }} $, and $\sigma_{\max}(A_{\tau_{k}})$ denotes the maximum singular value of $A_{\tau_{k}}$.
\end{theorem}
\begin{proof}
	 It follows from Lemma $\ref{thm0}$ that
	 $$
	 		D_f^{x_{k+1}^*}(x_{k+1},\hat{x})\leq D_{f}^{x_k^*}(x_k,\hat{x})-\frac{1}{2}\cdot \frac{(\eta_{k}^T(b-Ax_k))^2}{\|A^T\eta_{k}\|_2^2}.
	 $$
	 
	 Denote $E_k\in \mathbb{R}^{m\times [\tau_{k}]}$ by the matrix whose columns in turn are composed of all the vectors $e_i\in\mathbb{R}^{m}$ with $i \in \tau_{k}$, $[\tau_{k}]$ is the number of the index $\tau_{k}$, then $A_{\tau_{k}}=E_k^TA$. Set $\xi_k=E_k^T\eta_{k}$,
	 \begin{equation}
	 	\label{equ13}
	 	\left\|\xi_k\right\|_{2}^{2}=\eta_{k}^{T} E_{k} E_{k}^{T} \eta_{k}=\left\|\eta_{k}\right\|_{2}^{2}=\sum_{i \in \tau_{k}}\left|b_{i}-a_{i}^T x_{k}\right|^{2},
	 \end{equation}
	 and 
	$$
		\label{equ14}
		\left\|A^{T} \eta_{k}\right\|_{2}^{2}=\eta_{k}^{T} A A^{T} \eta_{k}=\xi_k^{T} E_{k}^{T} A A^{T} E_{k} \xi_k=\xi_k^{T} A_{\tau_{k}} A_{\tau_{k}}^{T} \xi_k=\left\|A_{\tau_{k}}^{T} \xi_k\right\|_{2}^{2}.
	$$
	 By making use of $\eqref{equ13}$, it is obtained that
	 $$
	 	\label{equ15}
	 	\begin{aligned} \eta_{k}^{T}\left(b-A x_{k}\right) & =\left(\sum_{i \in \tau_{k}}\left(b_{i}-a_{i}^T x_{k}\right) e_{i}^{T}\right)\left(b-A x_{k}\right) \\ & =\sum_{i \in \tau_{k}}\left(\left(b_{i}-a_{i}^T x_{k}\right) e_{i}^{T}\left(b-A x_{k}\right)\right) \\ & =\sum_{i \in \tau_{k}}\left|b_{i}-a_{i} ^Tx_{k}\right|^{2} \\ & =\left\|\xi_k\right\|_{2}^{2}.\end{aligned}
	$$
	 Therefore,
	 $$
	 	\frac{(\eta_{k}^T(b-Ax_k))^2}{\|A^T\eta_{k}\|_2^2}=\frac{\left\|\xi_k\right\|_{2}^{2}\cdot \left\|\xi_k\right\|_{2}^{2}}{\left\|A_{\tau_{k}}^{T} \xi_k\right\|_{2}^{2}}=\frac{\sum_{i\in \tau_{k}}\left(b_{i}-a_{i}^T x_{k}\right)^2\cdot \left\|\xi_k\right\|_{2}^{2}}{\left\|A_{\tau_{k}}^{T} \xi_k\right\|_{2}^{2}}.
	 $$
	 Because of the inequality $\left\|A_{\tau_{k}}^{T} \xi_k\right\|_{2}^{2}\leq \sigma_{\max}^2(A_{\tau_{k}})\|\xi_k\|_2^2$, 
	 $$
	 \begin{aligned}
	 	\frac{\sum_{i\in \tau_{k}}\left(b_{i}-a_{i}^T x_{k}\right)^2\cdot \left\|\xi_k\right\|_{2}^{2}}{\left\|A_{\tau_{k}}^{T} \xi_k\right\|_{2}^{2}}&\geq  \frac{\sum_{i\in \tau_{k}}\left(b_{i}-a_{i}^T x_{k}\right)^2}{\sigma_{\max}^2(A_{\tau_{k}})}\\
	 	&\geq \frac{\sum_{i\in \tau_{k}}\epsilon_{k}\|b-Ax_k\|_2^2\|a_i\|_2^2}{\sigma_{\max}^2(A_{\tau_{k}})}\\
	 	&=\frac{\epsilon_{k}\|A_{\tau_{k}}\|_F^2\|b-Ax_k\|_2^2}{\sigma_{\max}^2(A_{\tau_{k}})}
	 \end{aligned}
	 $$
	 Hence,
	 $$
	 \begin{aligned}
	 	D_f^{x_{k+1}^*}(x_{k+1},\hat{x})& \leq D_{f}^{x_k^*}(x_k,\hat{x})-\frac{1}{2}\cdot \frac{(\eta_{k}^T(b-Ax_k))^2}{\|A^T\eta_{k}\|_2^2}\\
	 	& \leq D_{f}^{x_k^*}(x_k,\hat{x})-\frac{1}{2}\cdot \frac{\epsilon_{k}\|A_{\tau_{k}}\|_F^2\|b-Ax_k\|_2^2}{\sigma_{\max}^2(A_{\tau_{k}})}.
	 \end{aligned}
	 $$
	 Due to $D_f^{x_k^*}(x_k,\hat{x})\leq \nu \|Ax_k-b\|_2^2$, so
	 $$
	 D_f^{x_{k+1}^*}(x_{k+1},\hat{x})\leq \left(1-\frac{\epsilon_{k} \|A_{\tau_{k}}\|_F^2}{2\cdot \nu\sigma_{\max}^2(A_{\tau_{k}})} \right)D_{f}^{x_k^*}(x_k,\hat{x}).
	 $$
	  For any nonzero vector $\tilde{y}\in \mathbb{R}^{[\tau_{k}]}$, it is obtained that
	 $$
	 \sigma^2_{\min}(A) \leq \tilde{\sigma}_{\min}(A)\leq \frac{\left(E_{k} \tilde{y}\right)^{T} A_{J} A_{J}^{T}\left(E_{k} \tilde{y}\right)}{\left(E_{k} \tilde{y}\right)^{T}\left(E_{k}\tilde{y}\right)}=\frac{\left(E_{k} \tilde{y}\right)^{T} A A^{T}\left(E_{k} \tilde{y}\right)}{\left(E_{k} \tilde{y}\right)^{T}\left(E_{k}\tilde{y}\right)}=\frac{\tilde{y}^{T} A_{\tau_{k}} A_{\tau_{k}}^{T} \tilde{y}}{\tilde{y}^{T}\tilde{y}} \leq\sigma_{\max}^2(A_{\tau_{k}}),
	 $$
	 and
	 $$
	 \epsilon_{k}\left\|A_{\tau_{k}}\right\|_{F}^{2}=\sum_{i \in \tau_{k}} \epsilon_{k}\left\|a_{i}\right\|_{2}^{2} \leq \sum_{i \in \tau_{k}} \frac{\left|b_{i}-a_{i}^T x_{k}\right|^{2}}{\left\|b-A x_{k}\right\|_{2}^{2}} \leq 1,
	 $$
	 so if $b\neq 0$, then $\hat{x}\neq 0$, 
	 $$
	 0<q_k=\frac{1}{2}\cdot\epsilon_{k}\|A_{\tau_{k}}\|_F^2\cdot \frac{\sigma^2_{\min}(A)}{\sigma_{\max}^2(A_{\tau_{k}})}\cdot \frac{|\hat{x}|_{\min}}{|\hat{x}|_{\min}+2\lambda}\leq \frac{1}{2}< 1.
	 $$

	 If $x_0=x_0^*=0$ and using inequality ($\ref{equ7}$), by induction, it can be concluded that
	 $$
	 \|x_k-\hat{x}\|_2\leq \prod_{j=0}^{k} \left(1-q_j\right)^{\frac{k}{2}}\sqrt{\left(2\lambda \|\hat{x}\|_1+\|\hat{x}\|_2^2\right)}.
	 $$
\end{proof}
It is observed from Theorem \ref{thm2} that $q_k$ is changing during the iteration process because of the adapative strategy. Hence, the following corollary is given for the linear convergence of the Algorithm \ref{agm3}.
\begin{corollary}
	\label{cor}
	 Given an initial vector $x_0=x_0^*=0$, assuming that $b\neq 0$ is in the range $\mathcal{R}(A)$ of $A$. Then, the iteration sequence $\{x_k\}_{k=0}^{\infty}$ of the Algorithm $\ref{agm3}$ converge to the unique solution $\hat{x}$ of the regularized basis pursuit problem $\eqref{equ1}$, it holds that 
	 $$
	 D_f^{x_{k+1}^*}(x_{k+1},\hat{x})\leq \left(1-\tilde{q}\right)D_{f}^{x_k^*}(x_k,\hat{x}).
	 $$
	 and
	 $$
	 \|x_k-\hat{x}\|_2\leq \left(1-\tilde{q}\right)^{\frac{k}{2}}\sqrt{\left(2\lambda \|\hat{x}\|_1+\frac{1}{2}\|\hat{x}\|_2^2\right)}.
	 $$
	 where $\tilde{q} =\frac{1}{2\nu \sigma^2_{\max}(A)\kappa^2(A)}$, $\kappa(A)=\frac{\sigma_{\max}(A)}{\sigma_{\min}(A)}$ denotes the condition number of $A$.
\end{corollary}
\begin{proof}
	From $\epsilon_k\|A\|_F^2\geq 1$, $\max\{m,n\}\sigma_{\min}^{2}(A_{\tau_{k}})\|A_{\tau_{k}}\|_F^2\leq \max\{m,n\}\sigma_{\max}^{2}(A_{\tau_{k}})$ and $\kappa(A)\geq\kappa(A_{\tau_{k}})$, it is seen that 
	$$
	q_k=\frac{\epsilon_{k} \|A_{\tau_{k}}\|_F^2}{ 2\cdot \nu \sigma_{\max}^2(A_{\tau_{k}})} \geq \frac{1}{2\nu \sigma_{\max}^2(A)\kappa^2(A)}=\tilde{q}.
	$$
\end{proof}

It is observed from Theorem $\ref{thm2}$ that if $\theta=0$, then $\epsilon_k \|A_{\tau_{k}}\|_F^2=\frac{\|A_{\tau_{k}}\|_F^2}{\|A\|_F^2}$, compared the contraction factor $q$ in Theorem $\ref{thm1}$ with that of Theorem $\ref{thm2}$ ,
$$
\frac{q}{q_k}=\frac{\|A_{\tau_{k}}\|_F^2}{\|A\|_F^2}\cdot \frac{\sigma_{\max}^2(A)}{\sigma_{\max}^2(A_{\tau_{k}})}.
$$
Therefore, the partial surrogate hyperplane sparse Kaczmarz method and the residual-based surrogate hyperplane sparse Kaczmarz method have the same contraction factor when $\tau_{k}=m$, $\frac{q}{q_k}=1$.

\section{Numerical Experiments}
\label{sec:4}
\quad In this section, numerical experiments are presented to verify the efficiency of the proposed methods for the applicaiton from random Gaussian matrices, SuiteSparse matrices and image reconstruction. The residual-based surrogate hyperplane sparse Kaczmarz method (denoted by ‘SHSKR’) is compared with the partial residual-based surrogate hyperplane sparse Kaczmarz method for different parameters $\theta$, denoted by SHSKPR($\theta$=1), SHSKPR($\theta$=0.5), and SHSKPR($\theta$=0), respectively.

The relative solution error (RSE) is defined as:
$$
\mathrm{RSE}=\frac{\left\|x_{k}-\hat{x}\right\|_{2}^{2}}{\left\|\hat{x}\right\|_{2}^{2}}.
$$

All the methods start from zero vector and terminates once the stopping tolerance $\mathrm{RSE}<10^{-6}$ is satisfied or the number of iteration steps exceed 100000. In the numerical experiments, $\lambda =1.5$ for the soft shrinkage function $S_{\lambda}(x)$ and the number of the nonzero elements in $\hat{x}$ is $0.01\times n$. 

\quad In Tables $\ref{tab1}$ and $\ref{tab2}$, the number of iteration steps (denoted as ‘IT’) and the elapsed computing time in seconds (denoted as ‘CPU’) are listed for the SHSKPR($\theta$=1), SHSKPR($\theta$=0.5), SHSKPR($\theta$=0) and SHSKR respectively when the size of the random Gaussian matrices varies.

\begin{table}[!htbp]
	\centering
	\caption{Numerical results for random overdetermined matrices}
	\label{tab1}
	\begin{tabular}{ccccccccc}
		\hline
		\multirow{2}{*}{$m \times n$} & \multicolumn{2}{c}{$2000\times 1000$} & \multicolumn{2}{c}{$3000\times 1500$} & \multicolumn{2}{c}{$4000 \times 2000$} & \multicolumn{2}{c}{$5000\times 2000$} \\ \cline{2-9} 
		& IT               & CPU                & IT               & CPU                & IT                & CPU                & IT               & CPU                \\ \hline
		SHSKPR($\theta=1$)            & 1681             & 0.9038             & 3356             & 5.8475             & 10634             & 33.3527            & 19037            & 91.2536            \\
		SHSKPR($\theta=0.5$)          & 85               & 0.0604             & 100              & 0.2589             & 350               & 1.4561             & 623              & 3.9041             \\
		SHSKPR($\theta=0$)            & 22               & 0.0160             & 41               & 0.0957             & 52                & 0.2072             & 66               & 0.4210             \\
		SHSKR                         & 20               & 0.0108             & 39               & 0.0646             & 45                & 0.1530             & 62               & 0.2795             \\ \hline
	\end{tabular}
\end{table}

\begin{table}[!httbp]
	\centering
	\caption{Numerical results for random underdetermined matrices}
	\label{tab2}
	\begin{tabular}{ccccccccc}
		\hline
		\multirow{2}{*}{$m \times n$} & \multicolumn{2}{c}{$1000\times 2000$} & \multicolumn{2}{c}{$1500\times 3000$} & \multicolumn{2}{c}{$2000 \times 4000$} & \multicolumn{2}{c}{$2500\times 5000$} \\ \cline{2-9} 
		& IT               & CPU                & IT               & CPU                & IT               & CPU                 & IT               & CPU                \\ \hline
		SHSKPR($\theta=1$)            & 7079             & 5.5684             & 13764            & 27.7832            & 25778            & 100.0145            & 57077            & 351.4208           \\
		SHSKPR($\theta=0.5$)          & 200              & 0.2017             & 430              & 1.1350             & 727              & 3.5158              & 1076             & 8.3482             \\ 
		SHSKPR($\theta=0$)            & 36               & 0.0325             & 86               & 0.2391             & 108              & 0.5167              & 299              & 2.3189             \\
		SHSKR                         & 31               & 0.0225             & 73               & 0.1565             & 97               & 0.3574              & 246              & 1.5068             \\ \hline
	\end{tabular}
\end{table}
From Tables $\ref{tab1}$ and $\ref{tab2}$, it is observed that the SHSKR method is the fastest and outperforms than the other three methods in terms of both iteration steps and elapsed time whether the random Gaussian matrices are overdetermined or underdetermined. 

In Figures $\ref{fig1}$ and $\ref{fig2}$, we depict the curves of the relative solution error versus the iteration steps for the SHSKPR($\theta$=1), SHSKPR($\theta$=0.5), SHSKPR($\theta$=0) and SHSKR methods respectively.
\begin{figure}[!htbp]
	\centering
	\subfigure[2000$\times$ 1000]{
		\includegraphics[width=7cm]{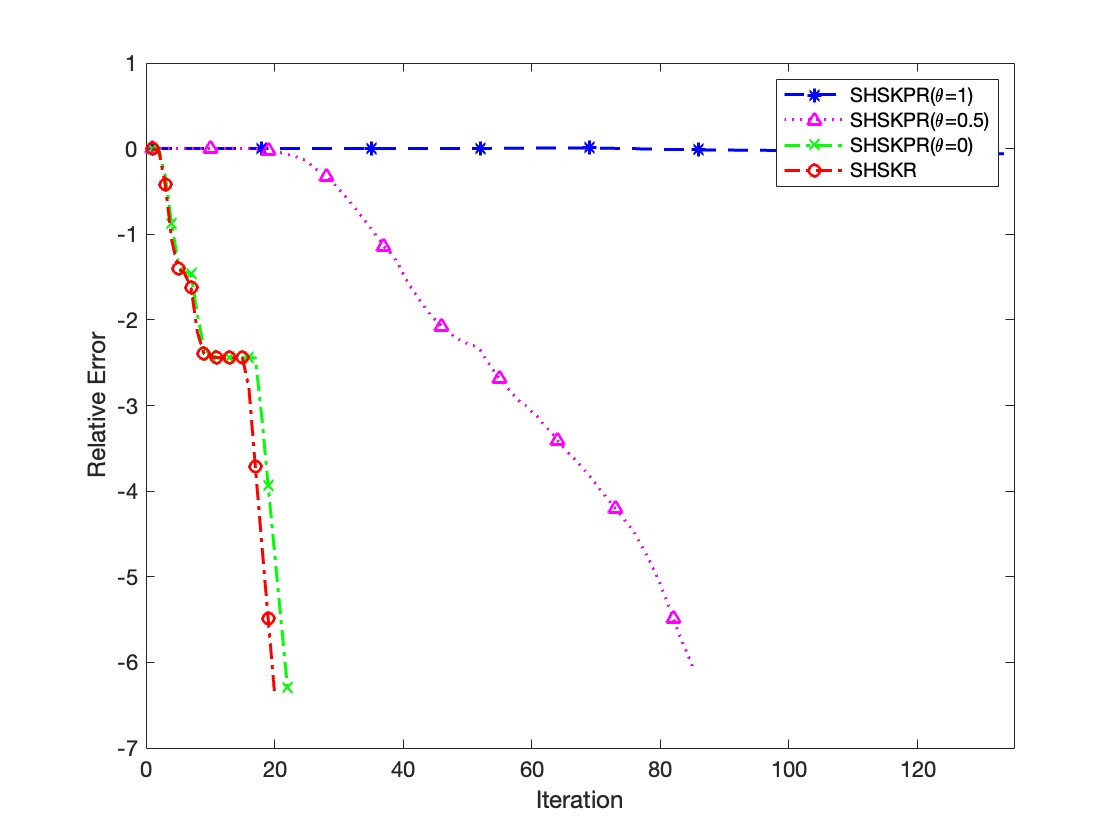}
	}
	\subfigure[3000$\times$ 1500]{
		\includegraphics[width=7cm]{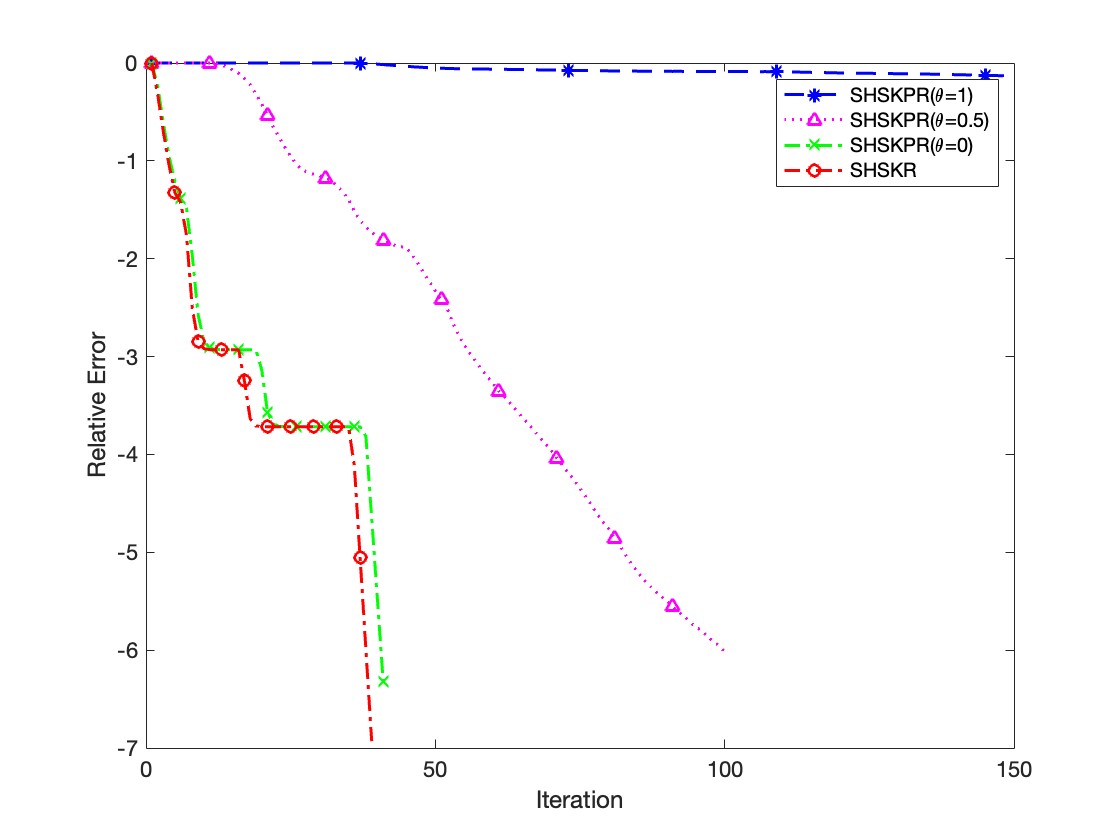}
	}
		\subfigure[4000$\times$ 2000]{
		\includegraphics[width=7cm]{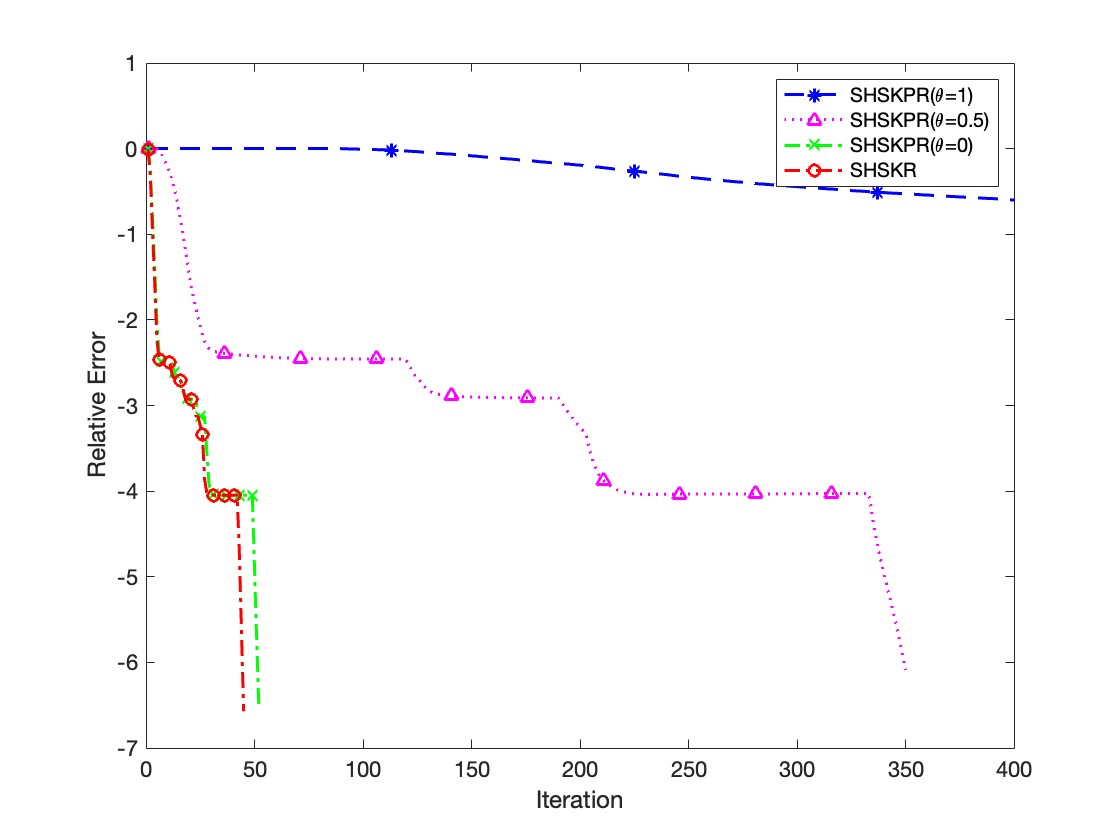}
	}
		\subfigure[5000$\times$ 2500]{
		\includegraphics[width=7cm]{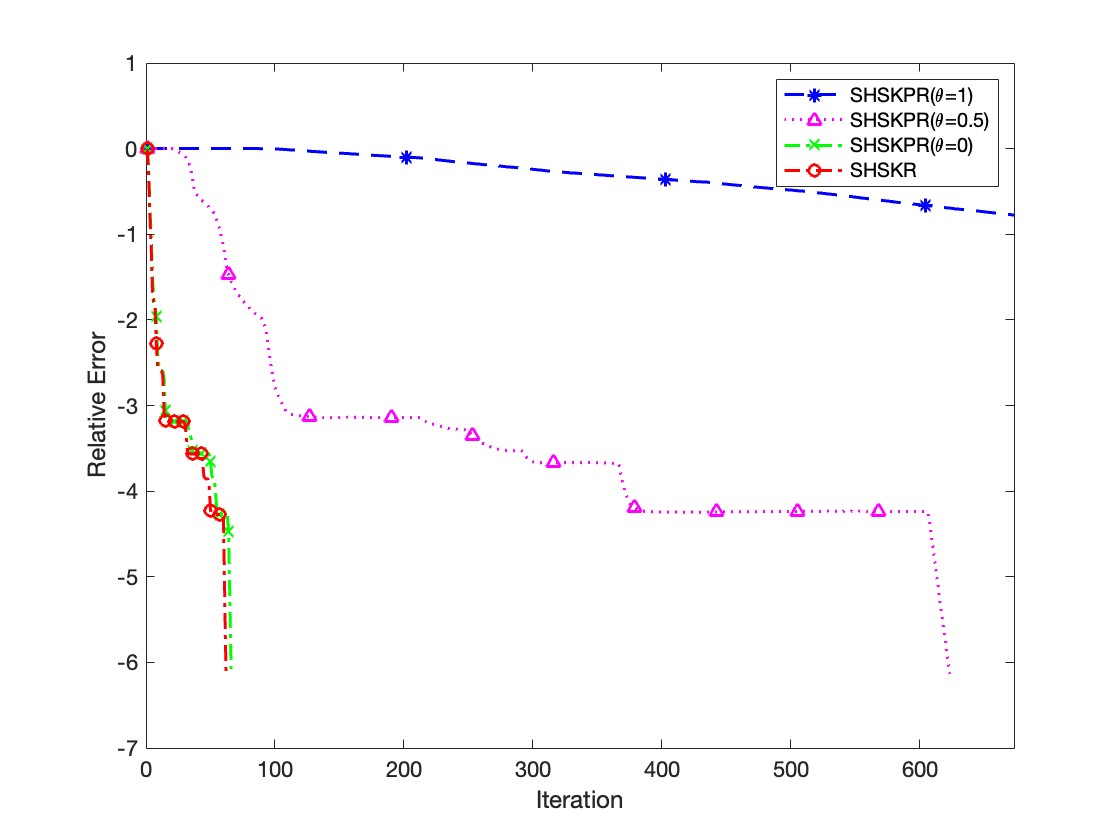}
	}
	\caption{Convergence curves for random overdetermined matrices.}
	\label{fig1}
\end{figure}

\begin{figure}[!htbp]
	\centering
	\subfigure[1000$\times$ 2000]{
		\includegraphics[width=7cm]{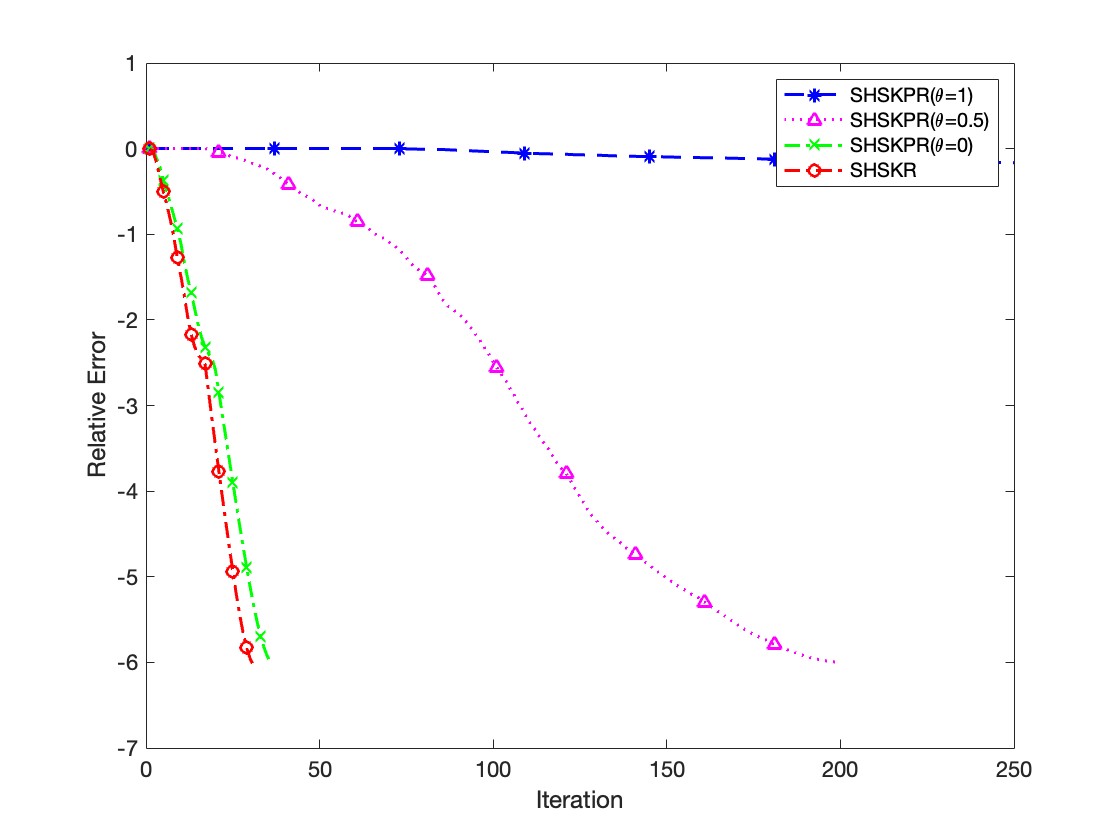}
	}
	\subfigure[1500$\times$ 3000]{
		\includegraphics[width=7cm]{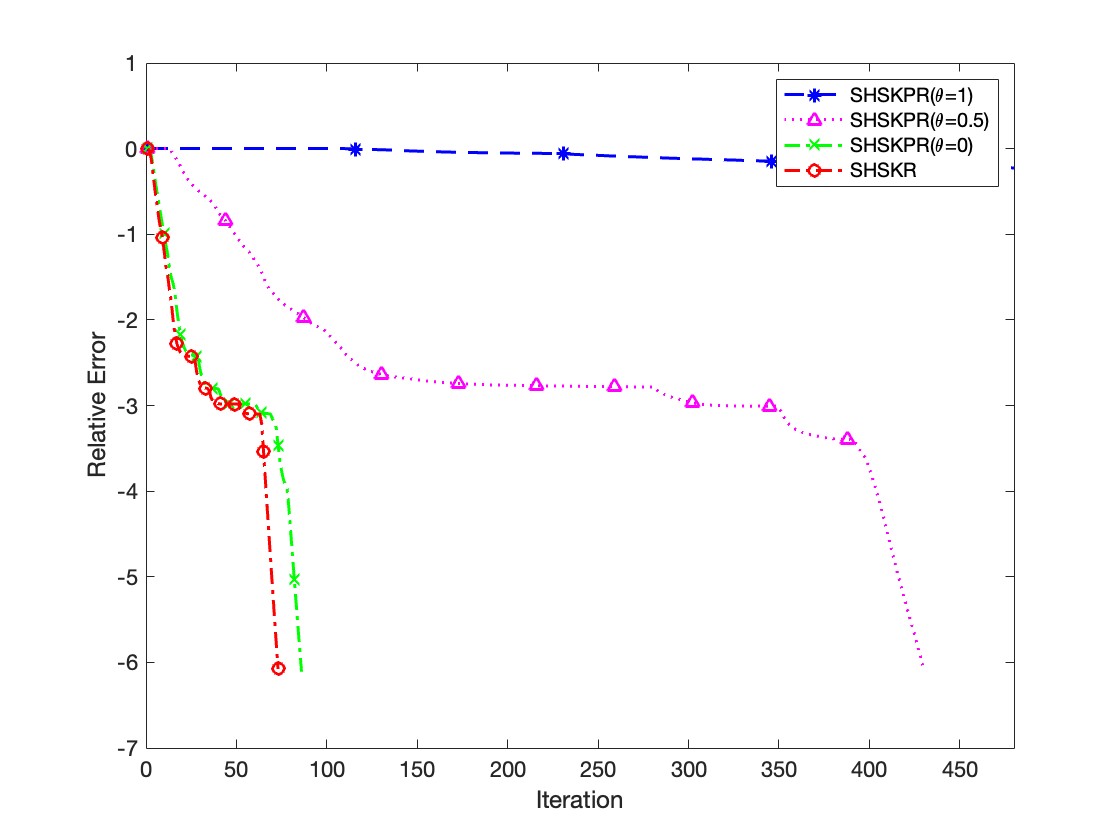}
	}
		\subfigure[2000$\times$ 4000]{
		\includegraphics[width=7cm]{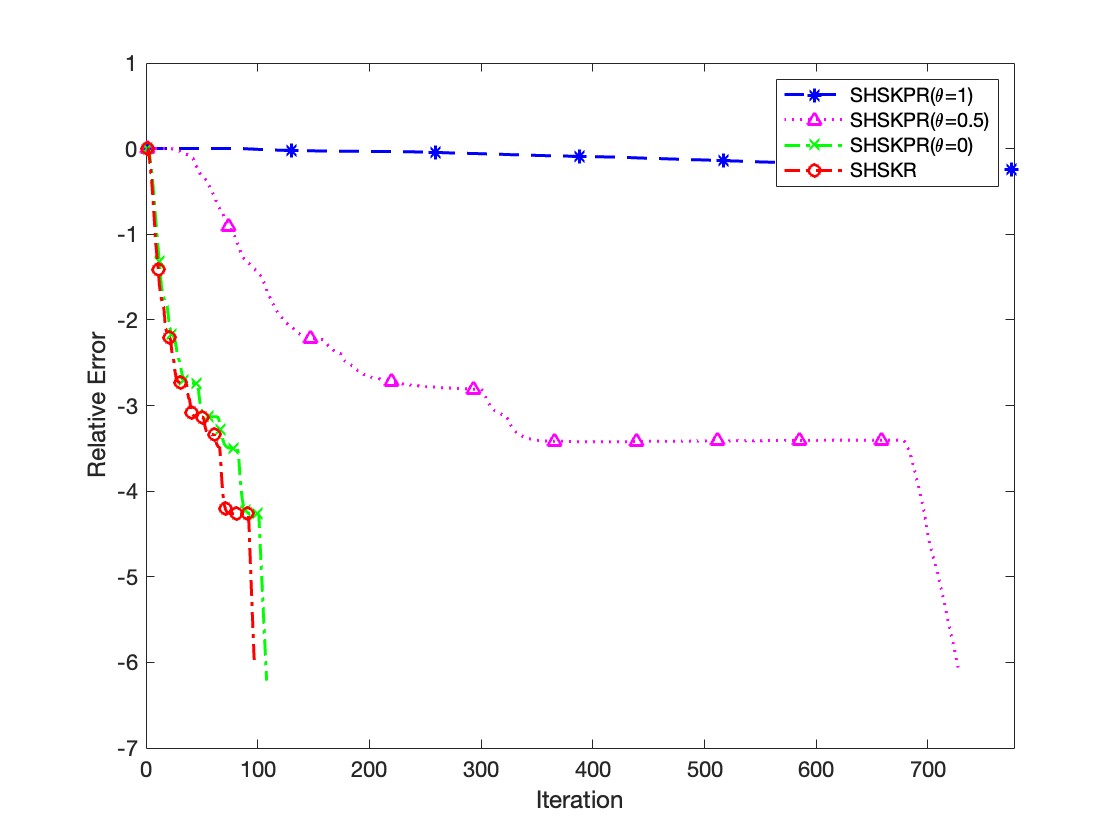}
	}
		\subfigure[2500$\times$ 5000]{
		\includegraphics[width=7cm]{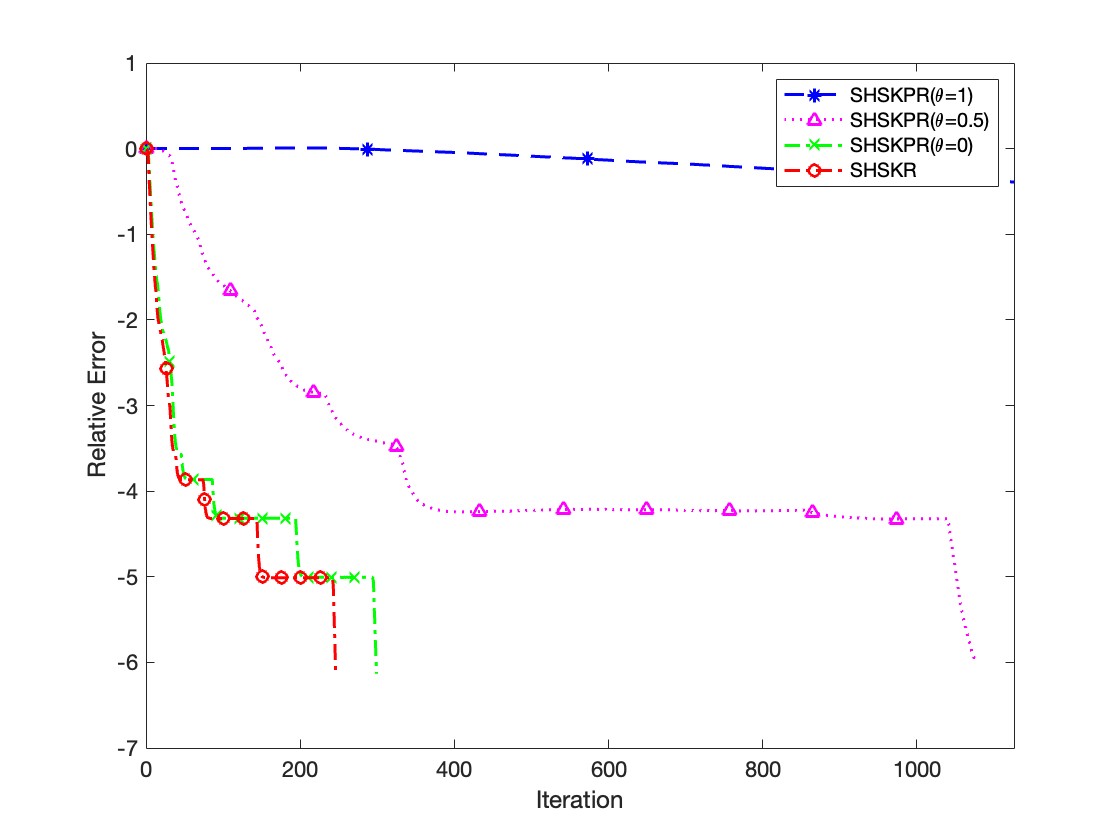}
	}
	\caption{Convergence curves for random underdetermined matrices.}
	\label{fig2}
\end{figure}
From Figures $\ref{fig1}$ and $\ref{fig2}$, it is seen that the SHSKPR($\theta$=1), SHSKPR($\theta$=0.5), SHSKPR($\theta$=0) and SHSKR methods converge successfully. Moreover, the SHSKR method requires the least number of iterations for all random matrices.

In Table $\ref{tab3}$, the properties of the SuiteSparse matrices including of dimensions, density, rank and condition number $\kappa(A)$ are listed.

\begin{table}[!htbp]
	\centering
	\caption{Matrices information}
	\label{tab3}
	\begin{tabular}{ccccc}
		\hline
		Name         & $m\times n$       & Density(\%) & Rank & $\kappa(A)$ \\ \hline
		bibd\_17\_3  & $136\times 680$   & 2.21        & 136  & 1.86        \\
		flower\_5\_1 & $211\times 201$   & 2.21        & 201  & Inf         \\
		str\_600     & $363\times 363$   & 2.40        & 363  & 190195.43   \\
		ash958       & $958\times 292$   & 0.68        & 292  & 3.20        \\
		well1850     & $1850\times 712$  & 0.65        & 712  & 111.31      \\
		illc1850     & $1850\times 712$  & 0.66        & 712  & 1404.90     \\
		relat6       & $2340\times 157$  & 2.21        & 157  & Inf         \\
		bibd\_81\_2  & $3240\times 3240$ & 0.03        & 3240 & 1.00           \\
		abtaha1      & $14596\times 209$ & 1.68        & 209  & 12.23       \\
		abtaha2      & $37932\times 331$ & 1.09        & 331  & 12.22       \\ \hline
	\end{tabular}
\end{table}

In Table $\ref{tab4}$, the number of iteration steps and elapsed time in seconds of the SHSKPR($\theta$=1), SHSKPR($\theta$=0.5), SHSKPR($\theta$=0) and SHSKR are reported respectively.

\begin{table}[!htbp]
	\centering
	\caption{Numerical results for matrices from SuiteSparse Matrix Collection.}
	\label{tab4}
	\begin{tabular}{ccccccccc}
		\hline
		\multirow{2}{*}{Name} & \multicolumn{2}{c}{SHSKPR($\theta=1$)} & \multicolumn{2}{c}{SHSKPR($\theta=0.5$)} & \multicolumn{2}{c}{SHSKPR($\theta=0$)} & \multicolumn{2}{c}{SHSKR} \\ \cline{2-9} 
		& IT                & CPU                & IT                & CPU                  & IT               & CPU                 & IT         & CPU          \\ \hline
		bibd\_17\_3           & 1349              & 0.0494             & 202               & 0.0101               & 122              & 0.0063              & 102        & 0.0034       \\
		flower\_5\_1          & 213               & 0.0070             & 86                & 0.0024               & 78               & 0.0022              & 72         & 0.0012       \\
		str\_600              & 648               & 0.0036             & 173               & 0.0017               & 90               & 0.0008              & 82         & 0.0004       \\
		ash958                & 80                & 0.0058             & 32                & 0.0040               & 24               & 0.0029              & 23         & 0.0011       \\
		well1850              & 241               & 0.0063             & 92                & 0.0046               & 46               & 0.0052              & 42         & 0.0020       \\
		illc1850              & 486               & 0.0079             & 134               & 0.0038               & 89               & 0.0025              & 79         & 0.0010       \\
		relat6                & 134               & 0.0144             & 24                & 0.0030               & 16               & 0.0018              & 15         & 0.0010       \\
		bibd\_81\_2           & 263               & 1.1925             & 150               & 0.8703               & 99               & 0.5799              & 95         & 0.4647       \\
		abtaha1               & 37                & 0.0283             & 11                & 0.0123               & 7                & 0.0081              & 7          & 0.0064       \\
		abtaha2               & 45                & 0.2118             & 27                & 0.1572               & 21               & 0.1276              & 20         & 0.0703       \\ \hline
	\end{tabular}
\end{table}

From Table $\ref{tab4}$, it is observed that the SHSKR method converges the fastest in terms of both iteration steps and elapsed time, compared with the other three methods. 

In Figure $\ref{fig3}$, the curves of the relative solution error versus the iteration steps for SHSKPR($\theta$=1), SHSKPR($\theta$=0.5), SHSKPR($\theta$=0) and SHSKR are plotted respectively.
\begin{figure}[!htbp]
	\centering
	\subfigure[	bibd\_17\_3]{
		\includegraphics[width=7cm]{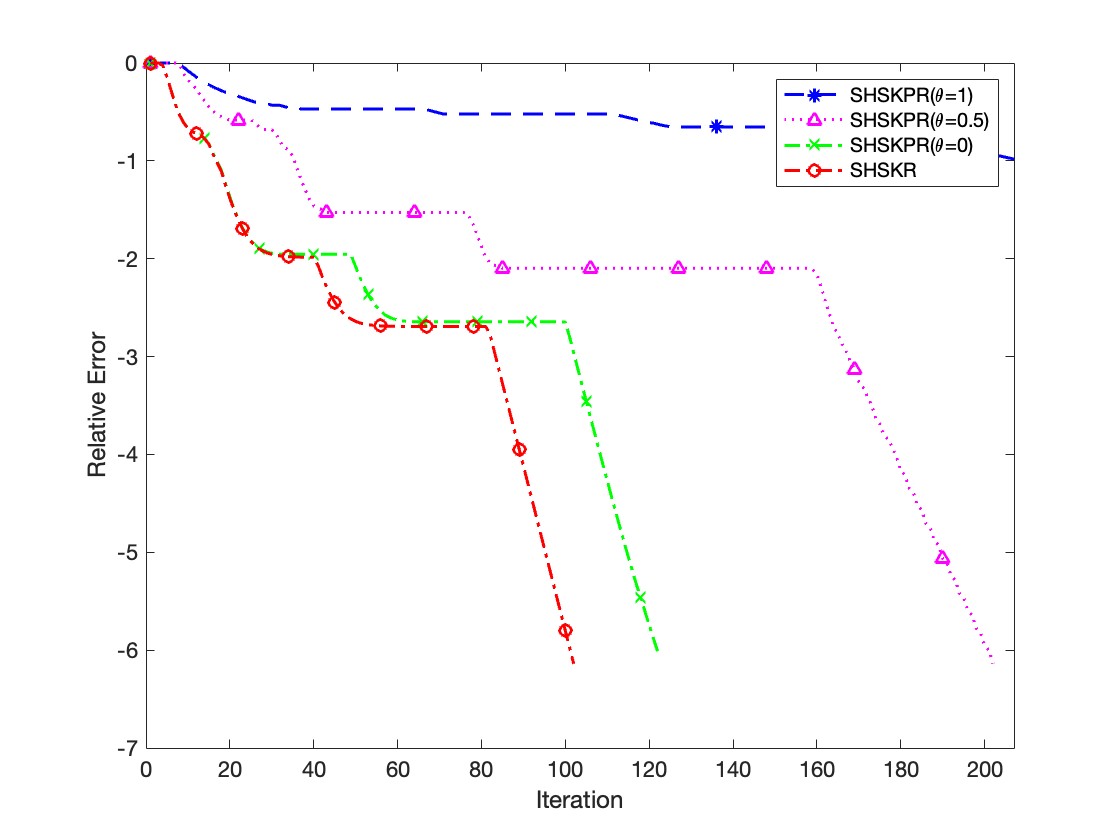}
	}
	\subfigure[well1850]{
		\includegraphics[width=7cm]{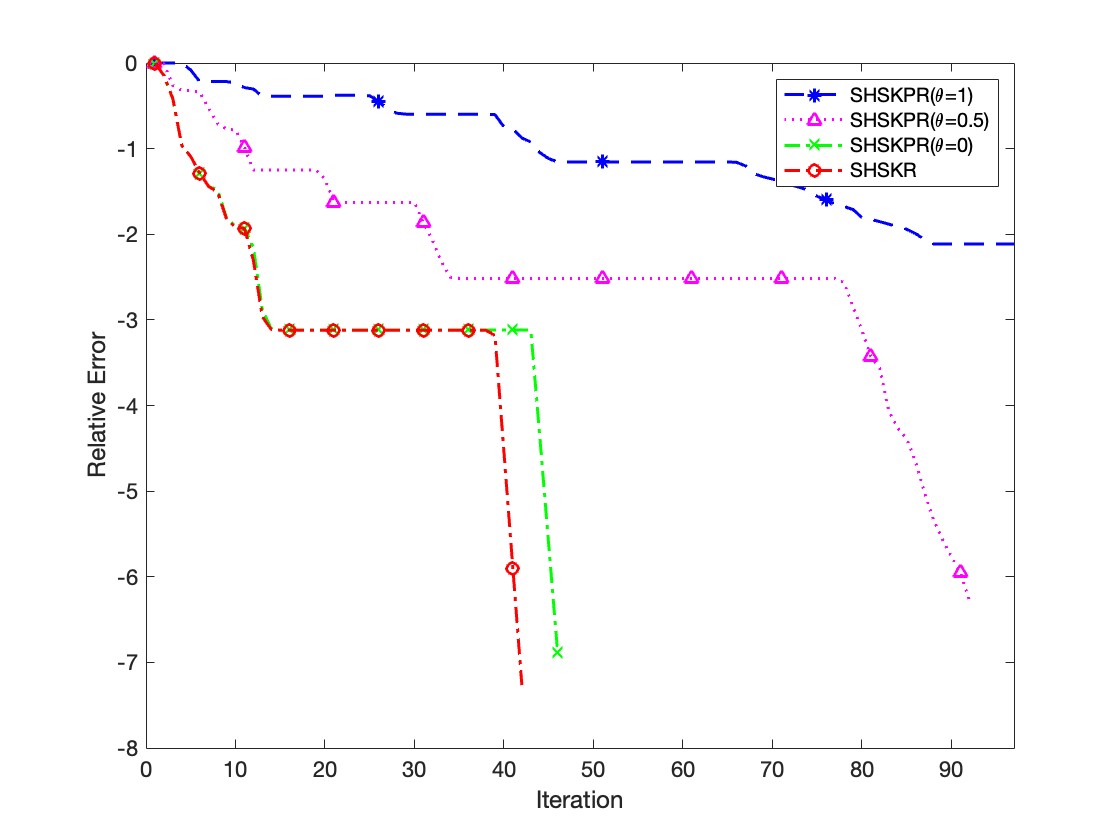}
	}
	\subfigure[relat6]{
		\includegraphics[width=7cm]{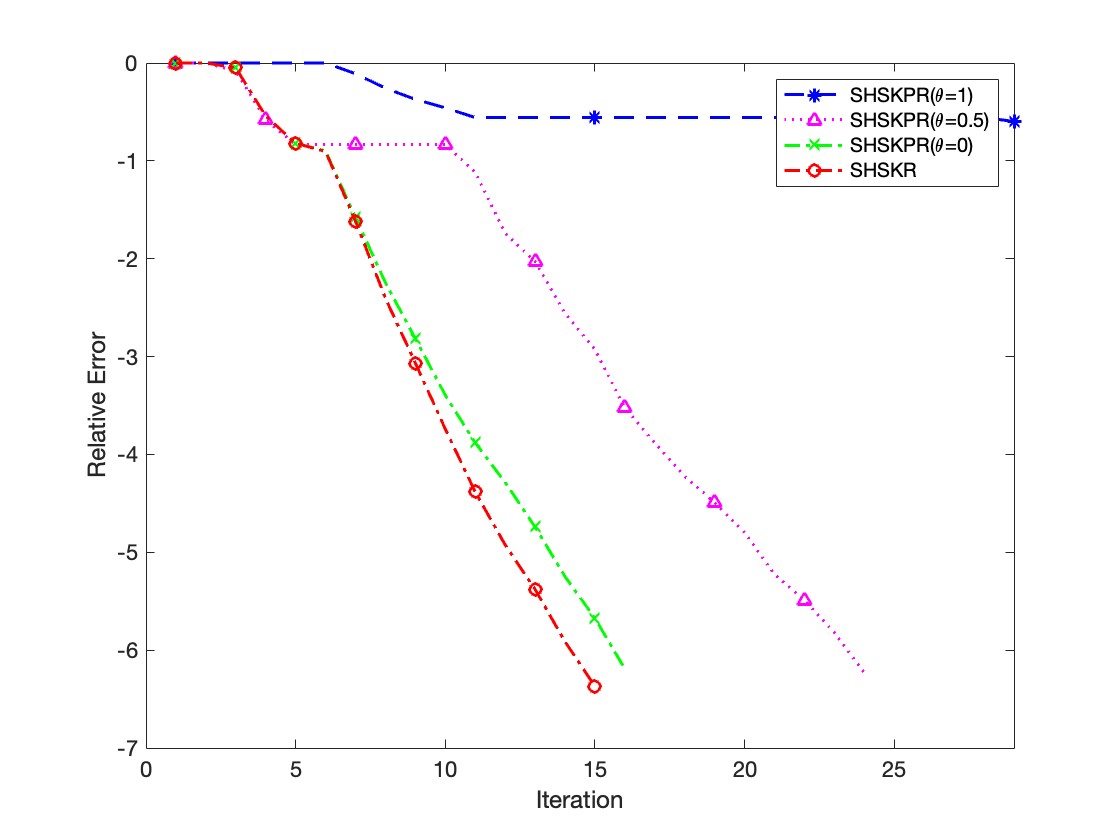}
	}
		\subfigure[str\_600]{
		\includegraphics[width=7cm]{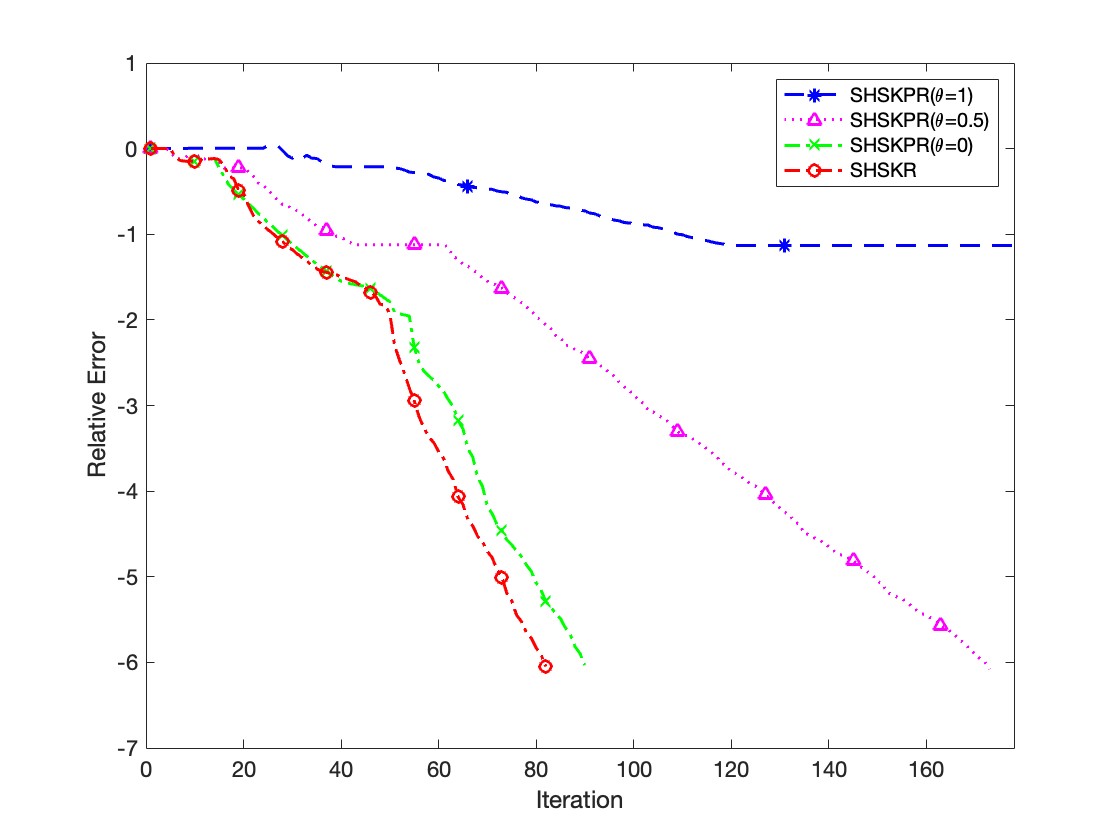}
	}
	\caption{Convergence curves for matrices from the SuiteSparse Matrix Collection.}
	\label{fig3}
\end{figure}

From Figure $\ref{fig3}$, it is seen that the convergence curve of SHSKR method is the fastest, which further confirm the eﬀiciency of the SHSKR method.

The test coefficient matrix is generated with the 2-D seismic travel-time tomography problem in AIR Tools II. The right term is $b=Ax_{*}+e$ where $e$ is the noise vector and the relative noise level is 0.01. The signal-noise ratio (SNR) is defined as 
$$
\mathrm{SNR}:=10 \log _{10} \frac{\sum_{i=1}^{n} x_{i}^{2}}{\sum_{i=1}^{n}\left(x_{i}-x_{*}\right)^{2}}.
$$
In Figure 4, the image reconstruction effect of four methods are plotted when they iterate for 200 times respectively. 

\begin{figure}[!htbp]
	\centering
	\subfigure[SHSKPR($\theta=1$), SNR=14.8280]{
		\includegraphics[width=6cm]{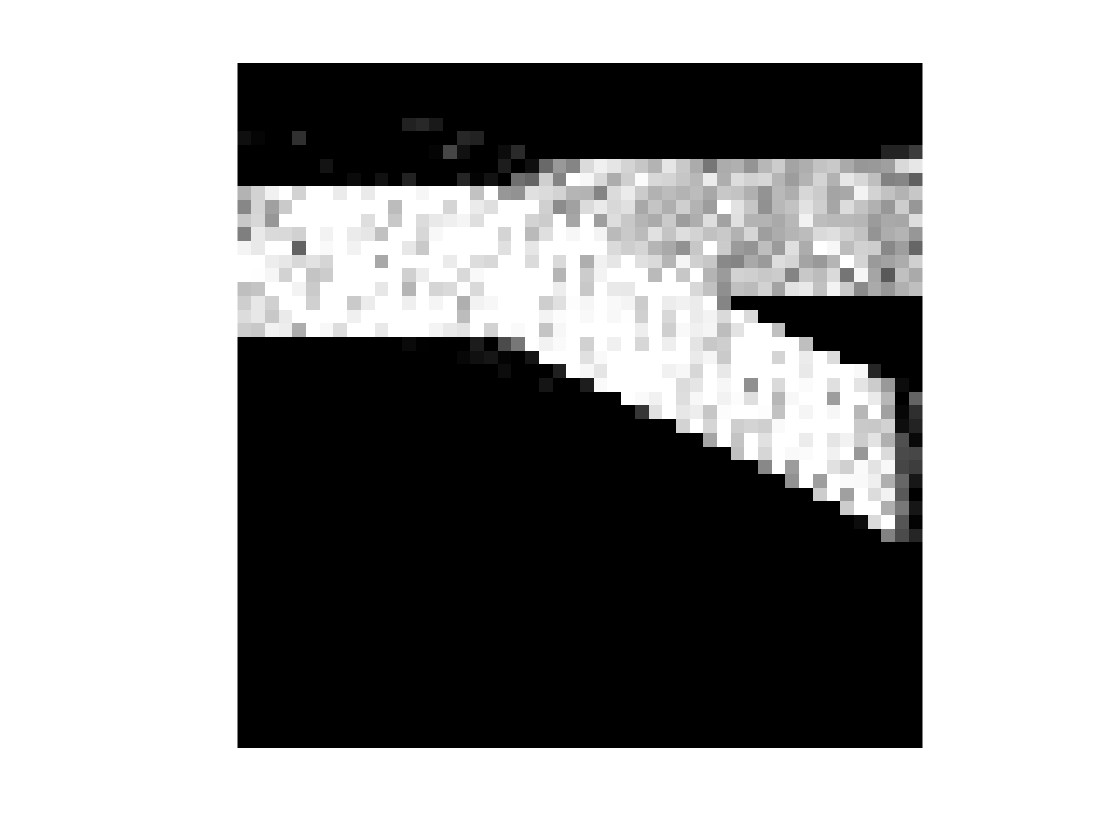}
	}
	\subfigure[SHSKPR($\theta=0.5$), SNR=19.8202]{
		\includegraphics[width=6cm]{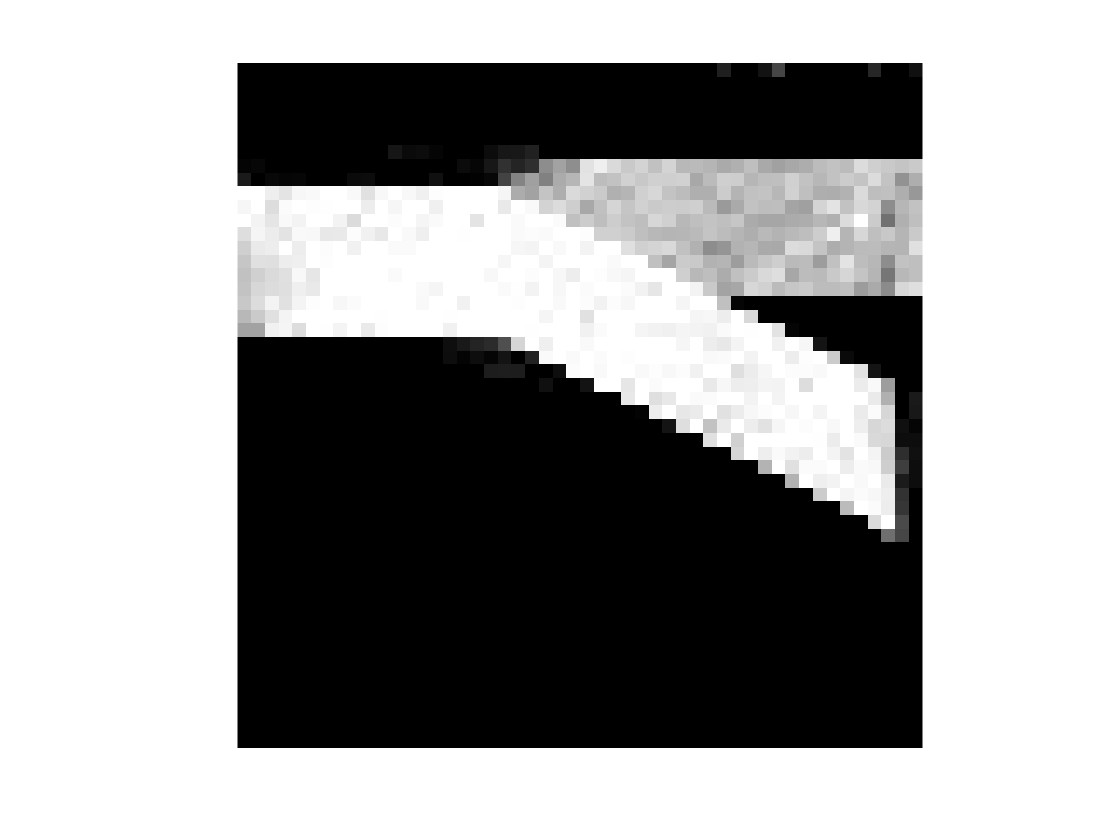}
	}
	\subfigure[SHSKPR($\theta=0$), SNR=23.2127]{
		\includegraphics[width=6cm]{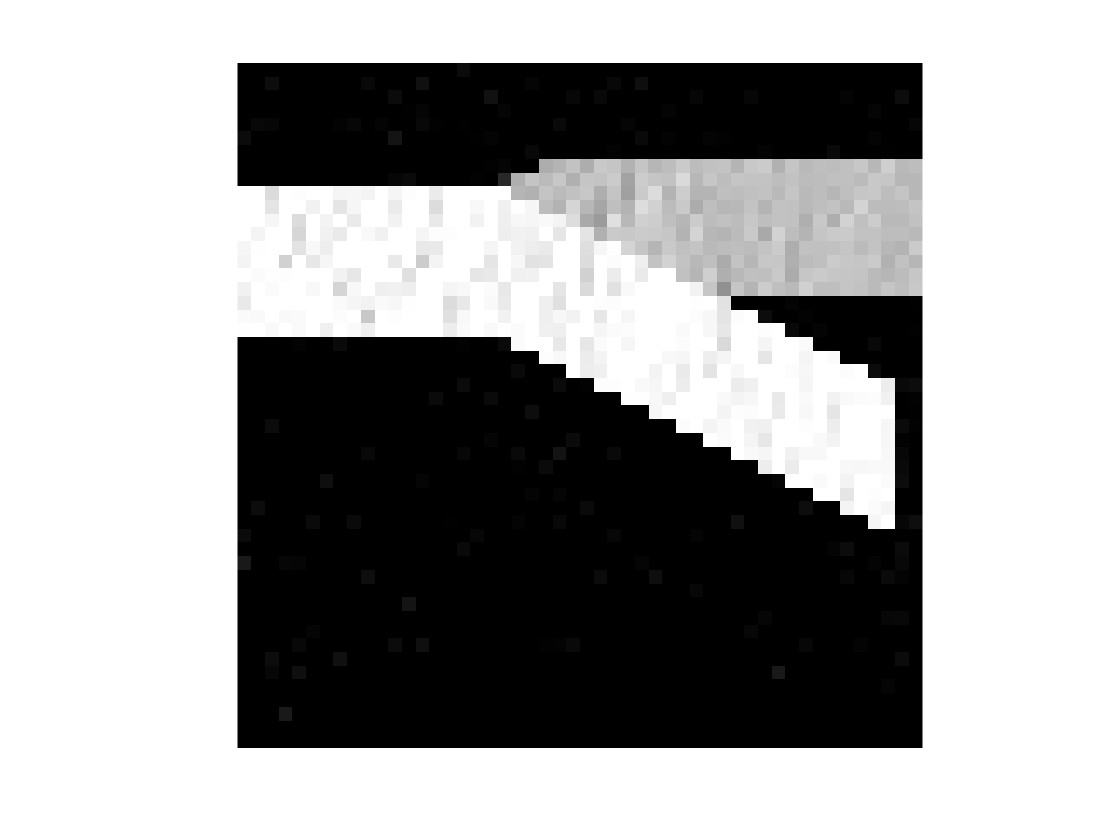}
	}
	\subfigure[SHSKR, SNR=23.9119]{
		\includegraphics[width=6cm]{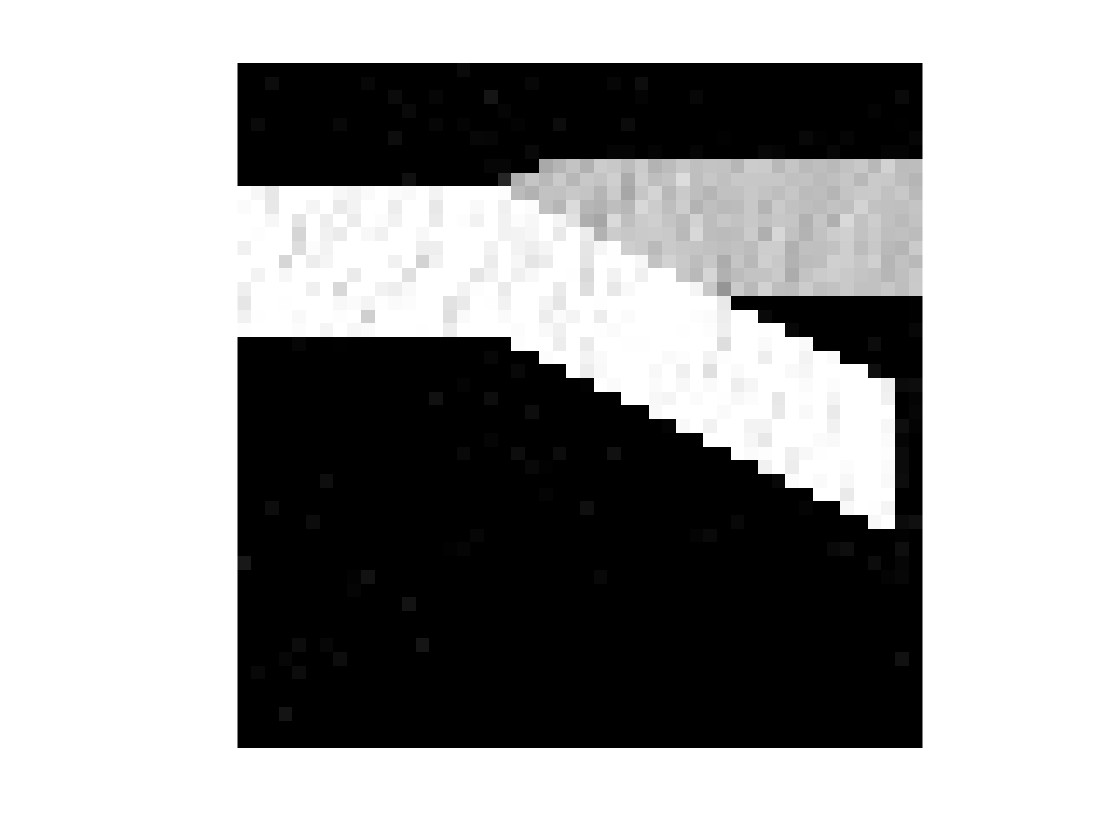}
	}
	\caption{Comparison of reconstruction results for 2-D seismic travel-time tomography.}
	\label{fig4}
\end{figure}
From Figure $\ref{fig4}$, it is seen that all four methods restore image eﬀiciently. In addition, the images reconstructed by the SHSKPR($\theta=0$) and SHSKR method are better than that of the other two methods from the view point of the sharpness of images and the values of SNR.

The matrix $A$ is generated with the VariantGaussianBlur in Restore Tools and the original numerical image is selected from the database in Matlab toolbox. In Figure $\ref{fig5}$, by using “VariantGaussianBlur” in Restore Tools to blur the original image, SHSKPR($\theta$=1), SHSKPR($\theta$=0.5), SHSKPR($\theta$=0) and SHSKR methods are plotted iterating 1000 steps to restore the number image from 0 to 9. 
\begin{figure}[!htbp]
	\centering
	\subfigure[original number image]{
		\includegraphics[width=15cm]{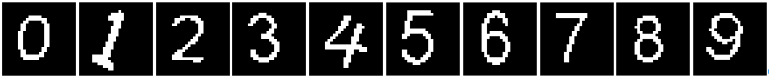}
	}
	\subfigure[blured number image]{
		\includegraphics[width=15cm]{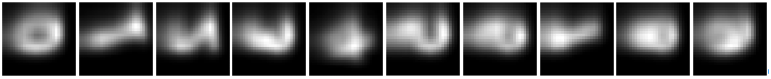}
	}
	\subfigure[SHSKPR($\theta=1$)]{
		\includegraphics[width=15cm]{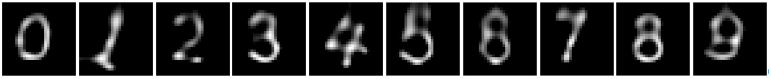}
	}
	\subfigure[SHSKPR($\theta=0.5$)]{
		\includegraphics[width=15cm]{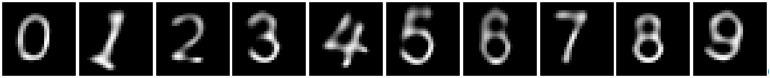}
	}
		\subfigure[SHSKPR($\theta=0$)]{
		\includegraphics[width=15cm]{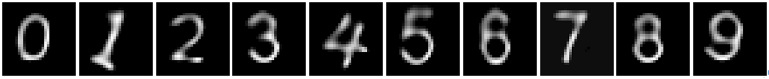}
	}
		\subfigure[SHSKR]{
		\includegraphics[width=15cm]{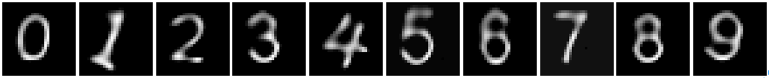}
	}
	\caption{Comparison of restoration results for the number ranging from 0 to 9.}
	\label{fig5}
\end{figure}

From Figure $\ref{fig5}$, it is seen that all four methods restore the blurred image successfully, where the SHSKR method achieves the best image restoration.

\section{Conclusion}
\label{sec:5}
\quad A surrogate hyperplane sparse Kaczmarz method is proposed for solving regularized basis pursuit problem. By making use of the residuals, the residual-based surrogate hyperplane sparse Kaczmarz method is proposed. In particular, the partial residual-based surrogate hyperplane sparse Kaczmarz method are introduced based on the relaxed and greedy criterion. The convergence theories of the new methods are established and studied in details. Numerical experiments further verify the efficiency of the proposed methods.

\bibliographystyle{unsrt}
\bibliography{ref}

\begin{thebibliography}{10}

\bibitem{donoho2006compressed}
D.~L. Donoho.
\newblock Compressed sensing.
\newblock {\em IEEE Transactions on information theory}, 52(4):1289--1306, 2006.

\bibitem{candes2006robust}
E.~J. Cand{\`e}s, J.~Romberg, and T.~Tao.
\newblock Robust uncertainty principles: Exact signal reconstruction from highly incomplete frequency information.
\newblock {\em IEEE Transactions on information theory}, 52(2):489--509, 2006.

\bibitem{byrd2012sample}
{R. H. Byrd, G. M. Chin, J. Nocedal, and Y. Wu}.
\newblock Sample size selection in optimization methods for machine learning.
\newblock {\em Mathematical programming}, 134(1):127--155, 2012.

\bibitem{strohmer2009randomized}
T.~Strohmer and R.~Vershynin.
\newblock A randomized {Kaczmarz} algorithm with exponential convergence.
\newblock {\em Journal of Fourier Analysis and Applications}, 15(2):262--278, 2009.

\bibitem{zeng2023adaptive}
Y.~Zeng, D.-R. Han, Y.-S. Su, and J.-X. Xie.
\newblock On adaptive stochastic heavy ball momentum for solving linear systems.
\newblock {\em arXiv preprint arXiv:2305.05482}, 2023.

\bibitem{han2024randomized}
D.-R. Han, Y.-S. Su, and J.-X. Xie.
\newblock Randomized douglas--rachford methods for linear systems: Improved accuracy and efficiency.
\newblock {\em SIAM Journal on Optimization}, 34(1):1045--1070, 2024.

\bibitem{bai2023randomized}
Z.-Z. Bai and W.-T. Wu.
\newblock Randomized kaczmarz iteration methods: Algorithmic extensions and convergence theory.
\newblock {\em Japan Journal of Industrial and Applied Mathematics}, 40(3):1421--1443, 2023.

\bibitem{bai2023convergence}
Z.-Z. Bai and L.~Wang.
\newblock On convergence rates of kaczmarz-type methods with different selection rules of working rows.
\newblock {\em Applied Numerical Mathematics}, 186:289--319, 2023.

\bibitem{ferreira2024survey}
I.~A. Ferreira, J.~A. Acebr{\'o}n, and J.~Monteiro.
\newblock Survey of a class of iterative row-action methods: The kaczmarz method.
\newblock {\em arXiv preprint arXiv:2401.02842}, 2024.

\bibitem{lorenz2014sparse}
D.~A. Lorenz, S.~Wenger, F.~Sch{\"o}pfer, and M.~Magnor.
\newblock A sparse {Kaczmarz} solver and a linearized {Bregman} method for online compressed sensing.
\newblock In {\em 2014 IEEE international conference on image processing (ICIP)}, pages 1347--1351. IEEE, 2014.

\bibitem{lorenz2014linearized}
D.~A. Lorenz, F.~Sch{\"o}pfer, and S.~Wenger.
\newblock The linearized {Bregman} method via split feasibility problems: analysis and generalizations.
\newblock {\em SIAM Journal on Imaging Sciences}, 7(2):1237--1262, 2014.

\bibitem{petra2015randomized}
S.~Petra.
\newblock Randomized sparse block {Kaczmarz} as randomized dual block-coordinate descent.
\newblock {\em Analele {\c{s}}tiin{\c{t}}ifice ale Universit{\u{a}}{\c{t}}ii "Ovidius" Constan{\c{t}}a. Seria Matematic{\u{a}}}, 23(3):129--149, 2015.

\bibitem{schopfer2019linear}
F.~Sch{\"o}pfer and D.~A. Lorenz.
\newblock Linear convergence of the randomized sparse {Kaczmarz} method.
\newblock {\em Mathematical Programming}, 173:509--536, 2019.

\bibitem{lei2018learning}
Y.-W. Lei and D.-X. Zhou.
\newblock Learning theory of randomized sparse {Kaczmarz} method.
\newblock {\em SIAM Journal on Imaging Sciences}, 11(1):547--574, 2018.

\bibitem{wang2021}
{Z. Wang and J.-F. Yin}.
\newblock Sparse greedy randomized {Kaczmarz} method for sparse solutions of linear equations.
\newblock {\em Journal of the Tongji University (Nature Science)}, 49(11):1505--1513, 2021.

\bibitem{lunglmayr2017microkicking}
{M. Lunglmayr and M. Huemer}.
\newblock Microkicking for fast convergence of sparse {Kaczmarz} and sparse lms.
\newblock In {\em 2017 IEEE 7th International Workshop on Computational Advances in Multi-Sensor Adaptive Processing (CAMSAP)}, pages 1--5. IEEE, 2017.

\bibitem{tondji2023acceleration}
L.~Tondji, I.~Necoara, and D.~A. Lorenz.
\newblock Acceleration and restart for the randomized {Bregman-Kaczmarz} method.
\newblock {\em arXiv preprint arXiv:2310.17338}, 2023.

\bibitem{yuan2022adaptively}
Z.-Y. Yuan, L.~Zhang, H.-X. Wang, and H.~Zhang.
\newblock Adaptively sketched {Bregman} projection methods for linear systems.
\newblock {\em Inverse Problems}, 38(6):065005, 2022.

\bibitem{tondji2023adaptive}
L.~Tondji, I.~Tondji, and D.~A. Lorenz.
\newblock Adaptive {Bregman}-{Kaczmarz}: An approach to solve linear inverse problems with independent noise exactly.
\newblock {\em arXiv preprint arXiv:2309.06186}, 2023.

\bibitem{gower2023nonlinear}
{R. Gower, D. A. Lorenz, and M. Winkler}.
\newblock A {Bregman-Kaczmarz} method for nonlinear systems of equations.
\newblock {\em arXiv preprint arXiv:2309.06186}, 2023.

\bibitem{lorenz2023minimal}
{D. A. Lorenz and M. Winkler}.
\newblock Minimal error momentum {Bregman-Kaczmarz}.
\newblock {\em arXiv preprint arXiv:2307.15435}, 2023.

\bibitem{yun2023fast}
{Z. Yun, D. Han, Y.-S. Su, and J.-X. Xie}.
\newblock Fast stochastic dual coordinate descent algorithms for linearly constrained convex optimization.
\newblock {\em arXiv preprint arXiv:2307.16702}, 2023.

\bibitem{Rockafellar1998VariationalA}
R.~T. Rockafellar and R.~J.~B. Wets.
\newblock Variational analysis.
\newblock In {\em Grundlehren der mathematischen Wissenschaften}, 1998.

\bibitem{wang2023surrogate}
Z.~Wang and J.-F. Yin.
\newblock A surrogate hyperplane {Kaczmarz} method for solving consistent linear equations.
\newblock {\em Applied Mathematics Letters}, 144:108704, 2023.

\bibitem{gower2015randomized}
R.~M. Gower and P.~Richt{\'a}rik.
\newblock Randomized iterative methods for linear systems.
\newblock {\em SIAM Journal on Matrix Analysis and Applications}, 36(4):1660--1690, 2015.

\bibitem{chen2022fast}
{Q.-J. Chen and Z.-D. Huang}.
\newblock On a fast deterministic block {Kaczmarz} method for solving large-scale linear systems.
\newblock {\em Numerical Algorithms}, 89(3):1007--1029, 2022.

\bibitem{tondji2023faster}
L.~Tondji and D.~A. Lorenz.
\newblock Faster randomized block sparse {Kaczmarz} by averaging.
\newblock {\em Numerical Algorithms}, 93(4):1417--1451, 2023.

\end{thebibliography}
\end{document}